\documentclass[11pt,reqno,a4paper]{amsart}
\usepackage{amsmath,amsthm,verbatim,amscd,amssymb,setspace,enumitem,hyperref}
\usepackage{exscale,color}

\renewcommand{\phi}{\varphi}
\renewcommand{\epsilon}{\varepsilon}
\newcommand{\N}{\mathbb{N}}
\newcommand{\Z}{\mathbb{Z}}

\newcommand{\R}{\mathbb{R}}
\newcommand{\C}{\mathbb{C}}
\renewcommand{\P}{\mathbb{P}}

\newcommand{\im}{\operatorname{im}}

\renewcommand{\P}{\mathbb{P}}
\newcommand{\vol}{\rm vol}
\newcommand{\p}{\partial}
\newcommand{\norm}[1]{\Vert #1 \Vert}

\newcommand{\Ric}{\operatorname{Ric}}

\providecommand{\norm}[1]{\lVert#1\rVert}

\newtheoremstyle{fancy}{}{}{\itshape}{}{\textbf\bgroup}{.\egroup}{ }{}
\newtheoremstyle{fancy2}{}{}{\rm}{}{\textbf\bgroup}{.\egroup}{ }{}

\theoremstyle{fancy}
\newtheorem{theorem}{Theorem}[section]
\newtheorem{lemma}[theorem]{Lemma}
\newtheorem{corollary}[theorem]{Corollary}
\newtheorem{prop}[theorem]{Proposition}
\newtheorem{conj}[theorem]{Conjecture}
\newtheorem{obs}[theorem]{Observation}

\theoremstyle{fancy2}
\newtheorem{definition}[theorem]{Definition}
\newtheorem{example}[theorem]{Example}
\newtheorem{remark}[theorem]{Remark}

\setlist{leftmargin=*}

\numberwithin{equation}{section}

\begin{document}
\title{Asymptotically conical Calabi-Yau manifolds, I}
\date{\today}
\author{Ronan J.~Conlon}
\address{D\'epartement de Math\'ematiques, Universit\'e du Qu\'ebec \`a Montr\'eal, Case Postale 8888, Succursale
Centre-ville, Montr\'eal (Qu\'ebec), H3C 3P8, Canada}
\email{rconlon@cirget.ca}
\author{Hans-Joachim Hein}
\address{Laboratoire de Math{\'e}matiques Jean Leray, Universit{\'e} de Nantes, 2 rue de la Houssini{\`e}re, 44322 Nantes Cedex 3, France}
\email{hansjoachim.hein@univ-nantes.fr}
\date{\today}
\begin{abstract}
This is the first part in a series of articles on complete Calabi-Yau manifolds asymptotic to Riemannian cones at infinity. We begin by proving  general existence and uniqueness results. The uniqueness part relaxes the decay condition $O(r^{-n-\epsilon})$ needed in earlier work to $O(r^{-\epsilon})$, relying on some new ideas about 
harmonic functions.
We then look at a few examples: (1) Crepant resolutions of cones. This includes a new class of Ricci-flat small resolutions associated with flag manifolds. (2) Affine deformations of cones. One focus here is the question of the precise rate of decay of the metric to its tangent cone. We prove that the optimal 
rate for the Stenzel metric on $T^*S^n$ is $-2\frac{n}{n-1}$.
\end{abstract}

\maketitle
\markboth{Ronan J.~Conlon and Hans-Joachim Hein}{Asymptotically conical Calabi-Yau manifolds, I}

\section{Introduction}

\subsection{Background} Consider a complete Riemannian manifold $(M, g)$ with ${\rm Ric} \geq 0$ and Euclidean volume growth. Cheeger-Colding theory \cite{ChNotes} implies that each blowdown sequence $(M, \lambda_i g)$, $\lambda_i \to 0$, subconverges in the pointed Gromov-Hausdorff sense to the cone over some length space. If ${\rm Ric} = 0$, {then} this cone is likely {to be} the same for all  sequences, so that $M$ would be
\emph{asymptotically} \emph{conical} in a Gromov-Hausdorff sense. This was proved by Cheeger and Tian \cite{Cheeger}, assuming $|{\rm Rm}| = O(r^{-2})$
(so that the links of all bona fide tangent cones are \emph{smooth}) and an integrability {condition}. In the course of {their} proof, they in fact establish $C^\infty$ convergence to one such cone at a rate of $O(r^{-\epsilon})$.

In this paper and its sequels, we wish to construct a careful theory of complete Ricci-flat K{\"a}hler manifolds that are asymptotically conical (\textquotedblleft AC\textquotedblright) in this stronger sense. If we think of noncompact
Calabi-Yau manifolds as bubbles in the singularity formation of
K{\"a}hler-Einstein manifolds, then at least heuristically, these are precisely the ones that correspond to noncollapsed isolated singularities. Besides \cite{Cheeger}, the foundational paper in this area is Tian-Yau 
\cite{Tian}, but see also \cite{goto, Joyce, vanC2}.

\begin{remark} Ricci-flat K{\"a}hler manifolds of Euclidean volume growth whose tangent cones are \emph{not} smooth, and in fact not even products of smooth cones, exist as well:
Joyce's QALE spaces \cite{Joyce} are desingularizations of very general flat orbifolds $\C^n/\Gamma$. Biquard-Gauduchon \cite{BG} wrote down explicit hyper-K{\"a}hler examples whose tangent cones are realized as nilpotent orbit closures in $\mathfrak{sl}(N,\C)$.
\end{remark}

The following three basic AC Ricci-flat examples are helpful to keep in mind throughout.

\begin{example}Calabi's metric on the total space of $K_{\P^{n-1}}$ \cite{Cal1}, whose tangent cone is $\C^n/\Z_n$. The underlying complex manifold is a \emph{crepant resolution} of the Calabi-Yau cone $\C^n/\Z_n$.
\end{example}

\begin{example} Stenzel's metric on $T^*S^n$ \cite{Stenzel}. Holomorphically, this is a \emph{deformation} of the ordinary double point in $\C^{n+1}$, containing a special Lagrangian $S^n$ rather than a holomorphic $\P^{n-1}$.
\end{example}

These two examples are related by {a} hyper-K{\"a}hler rotation if $n = 2$. If $n > 2$, then the first cone has no deformations. If $n > 3$, then the second cone has no crepant resolutions. Both manifolds are quasiprojective and may be compactified by adding positive divisors at infinity.

\begin{example}
Construct $M = X \setminus D$, 
where $X$ is a singular cubic $3$-fold $z_1 F(z) + z_2 G(z) = 0$ in $\P^4$ blown up along the plane $z_1 = z_2 = 0$, and $D$ denotes the proper transform of a hyperplane section avoiding the $4$ singularities. Then $M$ has a maximal compact analytic subset consisting of $4$ copies of $\P^1$. AC Ricci-flat K\"ahler metrics on $M$ were obtained in \cite{Tian}. 
In terms of singularity formation, this example represents a \emph{bubble tree}: The tangent cone to the singularity is the cone over a cubic surface; the deepest bubbles are small resolutions of the conifold (see \cite{delaossa} and Example \ref{ex:cdo}).
\end{example}

\subsection{Results}

The theoretical part of this paper (Sections \ref{s:existence}--\ref{s:uniqueness}) can be summarized as follows.\medskip\

\noindent {\bf Existence and uniqueness.} \emph{Let $M$ be an open complex manifold of complex dimension $n > 2$ with a holomorphic volume form $\Omega$. Let $(C, g_0)$ be a Calabi-Yau cone with holomorphic volume form $\Omega_0$ and let $r$ denote the radius function of $C$. Let $\Phi$ be a diffeomorphism
from $C$ to $M$, away from compact subsets of each, such that $\Phi^*\Omega - \Omega_0 = O(r^{-\epsilon})$ with respect to $g_0$ for some $\epsilon > 0$. If a class $\mathfrak{k} \in H^2(M)$ contains a K\"ahler form satisfying a mild asymptotic {condition}, then for all $c > 0$, $\mathfrak{k}$ also contains a unique Ricci-flat K{\"a}hler form $\omega_c$ such that $\Phi^*\omega_c - c\omega_0 = O(r^{-\delta})$ for some $\delta > 0$.} \medskip\

The existence part improves upon known results \cite{goto, Joyce, Tian, vanC3, vanC2}. Note that we do not require $\epsilon > 2$ or $\mathfrak{k} \in H^2_c(M)$, and that we exhibit the existence of a parameter $c$ in each K\"ahler class.

Uniqueness was known assuming that the two metrics in question are $O(r^{-n-\delta})$ close \cite{goto, Joyce, vanC3}, but this limits applications because AC Calabi-Yau metrics often decay more slowly than $O(r^{-n})$. 
We pursue a completely different approach here that allows us to relax $O(r^{-n-\delta})$ to $O(r^{-\delta})$.\medskip\

\noindent {\bf New idea to deal with uniqueness.} One version of this idea asserts the following: \emph{A harmonic function of strictly less than quadratic growth on a complete {\rm AC} K{\"a}hler manifold with ${\rm Ric} \geq 0$ is of  
necessity pluriharmonic} (Corollary \ref{c:subquadratic}). The proof makes heavy
use of the AC structure (the key ingredient here is a lemma from Cheeger-Tian \cite{Cheeger}); contrary to a claim in Li \cite{li}, the statement is in fact false for less than maximal volume growth.
The same idea turns out to be
useful in other settings as well, e.g.~when dealing with isolated conical singularities.\medskip\

\noindent {\bf Examples:~General picture.} Van Coevering \cite{vanC} pointed out that all AC K{\"a}hler manifolds can be made Stein by contracting compact analytic sets. We then have the following rough general picture, which we will flesh out in Sections \ref{s:crepres}--\ref{smoothing} by looking at some (extreme) special cases.

$\bullet$ If the complex structures on $M$ and $C$ converge sufficiently fast (say at least $O(r^{-2n})$), then $M$ carries a $b^2(M)$-dimensional family of AC Ricci-flat K{\"a}hler metrics of generic decay $r^{-2}$ towards the cone whose leading terms are harmonic $(1,1)$-forms on $C$.
This family then contains a distinguished $d$-dimensional subfamily of Ricci-flat metrics with leading term $i\partial\bar{\partial}r^{2-2n}$ whose K{\"a}hler forms lie in the image of $H^2_c(M) \hookrightarrow H^2(M)$. Here $d$ is the number of irreducible compact divisors in $M$.

$\bullet$ If the complex structures converge more slowly, {then} we still have a $b^2(M)$-dimensional family of metrics, but we know little about their fine asymptotics.
The problem of finding the precise decay rates, e.g.~in terms of the algebraic structure of $M$, does not seem to be a computational one.\medskip\

\noindent {\bf Crepant resolutions.} If $M$ is a crepant resolution of $C$, then $M$ and $C$ are biholomorphic outside a compact set. We begin Section \ref{s:crepres} by clarifying what is known about AC Ricci-flat metrics on such resolutions, where the general theory is {now} fairly complete
thanks to \cite{goto, Joyce, vanC3, vanC2}. 

We then apply our existence and uniqueness theory to obtain a new class of Ricci-flat metrics on certain \emph{small} resolutions of Calabi-Yau cones associated with partial flag manifolds.
For example, if $\ell = \frac{n}{n+1-k} \in \N$, then we will see that 
the total space of $Q^* \otimes (\det T)^{\ell}$ carries a $1$-parameter family of AC Ricci-flat metrics of rate $r^{-2}$ and {of} cohomogeneity $1$ under SU$(n+1)$, where $T$ denotes the rank
$k$ tautological bundle over the Grassmannian 
$G(k,n+1)$ and $Q = \underline{\C}^{n+1}/T$. This recovers Calabi's
hyper-K{\"a}hler metrics on $T^*\P^n$ \cite{Cal1} if $k = 1$, but our discussion here is entirely PDE-based.\medskip\

\noindent {\bf Affine smoothings.} At the opposite end of the spectrum,
we can assume that $M$ is {already} affine.
Then $M$ and $C$ are no longer biholomorphic at infinity, and finding the precise
decay rates becomes nontrivial; see Section \ref{smoothing}. Of course, $d = b^2_c(M) = 0$ {here}. Moreover, $b^2(M) = 0$ if $M$ is a complete intersection, so that we then obtain a \emph{unique} AC Calabi-Yau metric up to scaling. In this case, we present an algorithm that at least provides us with a {definite estimate} on the decay rate.

If $M$ is the standard smoothing of the ordinary double point in $\C^{n+1}$, we recover Stenzel's metric on $T^*S^n$ and our estimate of the rate is $-2\frac{n}{n-1}$. In this very particular example, we are then able to show that this rate is in fact \emph{optimal.} This contradicts Theorem 0.16 of Cheeger-Tian \cite{Cheeger}; compare Remark \ref{r:ct1} and {see Remark \ref{r:ct2}} for details. The theoretical parts of \cite{Cheeger} remain unaffected.
 
\subsection{Preliminaries}

\subsubsection{Riemannian cones} For us, {the definition of a Riemannian cone will take the following form}.

\begin{definition}\label{cone}
Let $( L, g)$ 
be a compact connected Riemannian manifold. The \emph{Riemannian cone} $C$ with link $L$ is defined to be $\R^+ \times L$ with metric $g_0 = dr^2 \oplus r^2g$ up to isometry. The radius function $r$ is then characterized intrinsically as the distance from the apex in the metric completion.
\end{definition}
 
Suppose that we are given a Riemannian cone $(C,g_{0})$ as above. Let $(r,x)$ be polar coordinates on $C$, where $x\in L$, and for $t>0$, define a map
$$\nu_{t}: L\times[1,2] \ni (r,x) \mapsto (tr,x) \in L \times [t,2t].$$ One checks that $\nu_{t}^{*}(g_{0})=t^{2}g_{0}$ and $\nu^{*}_{t}\circ\nabla_{0}=\nabla_{0}\circ\nu_{t}^{*}$, where $\nabla_{0}$ is the  Levi-Civita connection of $g_{0}$. Using these facts, one can prove the following basic lemma which will be useful in Section \ref{smoothing}.

\begin{lemma}\label{simple321}
Suppose that $\alpha\in\Gamma((TC)^{\otimes p}\otimes (T^{*}C)^{\otimes q})$ satisfies $\nu_{t}^{*}(\alpha)=t^{k}\alpha$ for every $t>0$ for some $k\in\R$. Then $|\nabla_{0}^{l}\alpha|_{g_{0}}=O(r^{k+p-q-l})$ for all $l\in\N_0$.
\end{lemma}

We shall say that ``$\alpha=O(r^{\lambda})$ with $g_{0}$-derivatives'' whenever $|\nabla_{0}^{k}\alpha|_{g_{0}}=O(r^{\lambda-k})$ for every $k \in \N_0$.
We will then also say that $\alpha$ has ``rate at most $\lambda$'', or sometimes, for simplicity, ``rate $\lambda$'', although it should be understood that (at least when $\alpha$ is purely polynomially behaved and does not contain any $\log$ terms) the rate of $\alpha$ is really the infimum of all $\lambda$ for which this holds.

\subsubsection{K{\"a}hler and Calabi-Yau cones}  Boyer-Galicki \cite{book:Boyer} is a comprehensive reference here.

\begin{definition}A \emph{K{\"a}hler cone} is a Riemannian cone $(C,g_0)$ such that $g_0$ is K{\"a}hler, together with a choice of $g_0$-parallel complex structure $J_0$. This will in fact often be unique up to sign. We then have a K{\"a}hler form $\omega_0(X,Y) = g_0(J_0X,Y)$, and $\omega_0 = \frac{i}{2}\p\bar{\p} r^2$ with respect to $J_0$.
\end{definition}

We call a K{\"a}hler cone ``quasiregular'' if the Reeb field $J\partial_r$ on its link generates an
$S^1$-action (and, in particular, ``regular'' if this $S^1$-action is free), and ``irregular'' otherwise.

\begin{theorem}\label{t:affine}
For every K{\"a}hler cone $(C,g_0,J_0)$, the complex manifold $(C,J_0)$ is isomorphic to the smooth part of a normal algebraic variety $V \subset \C^N$ with one singular point. In addition, $V$ can be taken to be invariant under a $\C^*$-action $(t, z_1,...,z_N) \mapsto (t^{w_1}z_1,...,t^{w_N}z_N)$ such that all $w_i > 0$.
\end{theorem}

This can be deduced from arguments written down by van Coevering in \cite[\S 3.1]{vanC4}.

\begin{definition}\label{d:cycone}
We say that $(C,g_0,J_0,\Omega_0)$ is a \emph{Calabi-Yau cone} if
\begin{enumerate}[label=\textnormal{(\roman{*})}, ref=(\roman{*})]
\item $(C,g_0, J_0)$ is a Ricci-flat K\"ahler cone of complex dimension $n$, 
\item the canonical bundle $K_{C}$ of $C$ with respect to $J_0$ is trivial, and
\item $\Omega_{0}$ is a nowhere vanishing section of $K_{C}$ with $\omega_0^n = i^{n^2}\Omega_0 \wedge \bar{\Omega}_0$.
\end{enumerate}
\end{definition}

We defer any discussion of Ricci-flat {K\"ahler} manifolds with torsion canonical bundle to \cite{ch2}.
  
\subsubsection{Calabi ansatz}\label{s:calans} The construction of Calabi-Yau cones, or Sasaki-Einstein manifolds, is in itself a highly nontrivial problem. See Sparks \cite{Sparks-survey} for an excellent recent survey.

 The most elementary construction, originating {in} Calabi's paper \cite{Cal1}, states that \emph{regular} Calabi-Yau cones are classified by K{\"a}hler-Einstein Fano manifolds; see LeBrun \cite[Proposition 3.1]{ Lebrun}. This idea will be useful for us in Section \ref{smoothing}, but see also Example \ref{ex:calabi1} for Calabi's original application.

The essence of the Calabi ansatz can be summarized as follows: If $D$ is K{\"a}hler-Einstein Fano, then for every integer $k > 0$ dividing $c_1(D)$, there exists a regular Calabi-Yau cone structure on the space $(\frac{1}{k}K_{D})^{\times}$, the blowdown of the zero section of the $k$-th root of $K_{D}$. If $D$ has dimension $n-1$, then the radius function of the Ricci-flat cone metric is given by $r^{n} = \norm{\cdot}_k^k$, where $\norm{\cdot}_k$ denotes the Hermitian norm on $\frac{1}{k}K_D$ naturally induced from the K\"ahler-Einstein metric on $D$.

\begin{definition}
The \emph{Fano index} of a Fano manifold $D$ is the divisibility of $K_D$ in ${\rm Pic}(D)$. 
\end{definition}

This is often $1$ and never more than $n$, with equality iff $D = \P^{n-1}$. The quadric in $\P^n$ ($n > 2$) is the only Fano of index $n-1$, and the cases of index $n-2$ {and} $n-3$ are classified as well \cite{AG5}.

\subsubsection{Asymptotically conical Riemannian manifolds} 

\begin{definition}\label{d:AC}
Let $(M,g)$ be a complete Riemannian manifold and let $(C,g_0)$ be a Riemannian cone. We call $M$ \emph{asymptotically conical} (AC) with tangent cone $C$ if there exists a diffeomorphism $\Phi: C\setminus K \to M \setminus K'$ with $K,K'$ compact, such that $\Phi^*g - g_0 = O(r^{-\epsilon})$ with $g_0$-derivatives for some $\epsilon > 0$. A \emph{radius function} is a smooth function $\rho: M \to [1,\infty)$ with $\Phi^*\rho = r$ away from $K'$.
\end{definition}

We implicitly only allow for one end in this definition. This is simply to fix ideas and because, {by the splitting theorem}, AC manifolds with ${\rm Ric} \geq 0$ can only {ever} have one end anyway.

The notions of {an} AC K{\"a}hler and {an} AC Calabi-Yau manifold are self-explanatory, except for the
{caveat} that we require a Calabi-Yau to have trivial canonical bundle, just as in Definition \ref{d:cycone}.

\subsubsection{K{\"a}hler classes on open complex manifolds} 
In this paper, we are really only interested in AC K{\"a}hler manifolds. These are $1$-convex, hence biholomorphic to resolutions of affine analytic varieties with at worst finitely many {normal} singularities; see Appendix \ref{s:ddbar}.

 It seems to be useful to think of a K{\"a}hler class on a $1$-convex manifold $M$ simply as a class in the de Rham group $H^2(M)$ that contains positive $(1,1)$-forms. For all we know, these classes may very well be characterized by the inequalities
$\int_Y \omega^k > 0$ for all $k$-dimensional ($k > 0$) compact analytic subsets $Y \subset M$, as is of course the case on compact K{\"a}hler manifolds \cite{DP}.

Our construction of AC Calabi-Yau metrics relies on a slightly more restrictive notion of K{\"a}hler 
class than this; compare Definition \ref{d:acsupp}. However, in all the examples that we discuss in this paper, this more restrictive definition turns out to in fact agree with the naive one above.

Let us finally remark that, in all the AC examples of interest here, the link $L$ satisfies $H^1(L) = 0$. In particular, by a standard long exact sequence, (\ref{e:basic_LES}), $H^2_c(M)$ injects as a subspace into $H^2(M)$. {Thus}, there is no ambiguity in talking about compactly supported classes in $H^2(M)$.

\subsection{Acknowledgments} This paper developed out of a part of RJC's doctoral dissertation \cite{RonanThesis} at Imperial College London. He would like to thank his supervisor, Mark Haskins, for his guidance and constant support and encouragement, and Gilles Carron and Dominic Joyce for useful remarks on a preliminary version of {his} thesis. He also wishes to acknowledge support from a Britton postdoctoral fellowship which he held at McMaster University while carrying out part of this research. HJH would like to acknowledge postdoctoral support under EPSRC Leadership Fellowship EP/G007241/1. We thank MPIM and HIM Bonn for excellent working conditions during our stay in Bonn in Fall 2011, the members of the Imperial geometry group for many helpful conversations, and Jeff Cheeger and Gang Tian for their willingness to discuss \cite{Cheeger}. Finally, we must acknowledge an intellectual debt of gratitude to Craig van Coevering, whose articles \cite{vanC, vanC3, vanC2, vanC4} inspired much of our research.
\newpage
\section{Existence for the complex Monge-Amp{\`e}re equation}\label{s:existence}

\subsection{Overview} We prove an abstract existence result for the complex Monge-Amp{\`e}re equation on AC K{\"a}hler manifolds. The existence of AC Calabi-Yau metrics can be deduced as a corollary.

The statement of the analytic result is as follows. We use standard notation concerning weighted H{\"o}lder spaces; cf.~Section \ref{s:linear_analysis}. Parts (i) and (ii) can be assembled from the literature, as we explain in Section \ref{s:fast_decay}. Part (iii) is new, and will be proved in Section \ref{s:slow_decay} by reduction to Part (ii).

\begin{theorem}\label{t:existence}
Let $(M, g, J)$ be an {\rm AC} K{\"a}hler manifold of complex dimension $n \geq 2$ with tangent cone $(C, g_0, J_0)$. Let $\omega$, $\omega_0$ denote the associated K{\"a}hler forms. {Denote by} $\mathcal{P}$ the set of exceptional weights of the scalar Laplacian as defined in {\rm (\ref{hothead})}, and {let $\rho$ be} a radius function on $M$. We wish to solve the complex Monge-Amp{\`e}re equation $(\omega + i\partial\bar{\partial}u)^n = e^f\omega^n$, where $f \in C^\infty_\beta(M)$.
\begin{enumerate}
\item If $\beta \in (-2n-\delta,-2n)$, $0 < \delta \ll 1$, then there exists a unique solution $u \in \R \rho^{2-2n} \oplus C^\infty_{\beta+2}(M)$.

\item If $\beta \in (-2n,-2)$, then there exists a unique solution $u \in C^\infty_{\beta+2}(M)$.

\item If $\beta \in (-2,0)$ and $\beta + 2 \not\in \mathcal{P}$, then there exists a solution $u \in C^\infty_{\beta+2}(M)$.
\end{enumerate}
\end{theorem}

\begin{remark}\label{r:existence}
(i) In the first case, the constant in front of $\rho^{2-2n}$ can be determined by integrating the identity $(e^f - 1)\omega^n = \frac{1}{2}(\Delta u)\omega^n + {n \choose 2} (i\partial\bar{\partial}u)^2 \wedge \omega^{n-2} + ... + (i\partial\bar{\partial}u)^n$ over $M$.

(ii) If $\Ric(g_0) \geq 0$, then we can take $\delta = 1$; compare Remark \ref{r:ricci_weights} below.
Moreover, if ${\rm Ric}(g) \geq 0$ on $M$ and $\beta \in (-2n,-2)$, then our work in Section \ref{s:uniqueness} implies that $u \in C^\infty_{\beta + 2}(M)$ is in fact the only solution in $C^\infty_{2-\epsilon}(M)$ for  all $0 < \epsilon \leq |\beta|$ up to adding on globally defined pluriharmonic functions.
\end{remark}

We now explain the consequences for the existence of Calabi-Yau metrics. This result encompasses all
previous such results in the literature, and in fact more. We will prove it in Section \ref{s:cy_ex_proof}.

\begin{definition}\label{d:acsupp}
Let $M$ be an open complex manifold, $K \subset M$ a compact set, $C = \R^+ \times L$ a cone with cone metric $g_0$, and $\Phi: (1,\infty) \times L \to M\setminus K$ a diffeomorphism. A class in $H^2(M)$ is said to be a \emph{$\mu$-almost compactly supported K{\"a}hler class} for some given $\mu < 0$ if it can be represented by a K{\"a}hler form $\omega$ on $M$ such that $\omega - \xi = d\eta$ on $M \setminus K$, with {$\eta$} a smooth real $1$-form on $M\setminus K$ and {$\xi$} a smooth real $(1,1)$-form on $M \setminus K$ such that $\Phi^*\xi = O(r^{\mu})$ with $g_0$-derivatives.
\end{definition}

\begin{theorem}\label{thm:main}
Let $(M,J)$ be an open complex manifold of complex dimension $n > 2$ such that $K_M$ is trivial.
Let $\Omega$ be a nowhere vanishing holomorphic volume form on $M$. Let $L$ be Sasaki-Einstein with associated Calabi-Yau cone $(C, g_0, J_0, \Omega_0)$ and radius function $r$. Suppose that there exist $\lambda < 0$, a compact subset $K \subset M$, and a diffeomorphism $\Phi: (1,\infty) \times L \to M\setminus K$ such that
\begin{equation} \Phi^{*}\Omega-\Omega_{0} =O(r^{\lambda})\end{equation}
with $g_0$-derivatives. Let $\mathfrak{k} \in H^2(M)$ be a $\mu$-almost compactly supported K{\"a}hler class. Assume that
\begin{equation}\nu := \max\{\lambda,\mu\}  \not\in \{-2n, -2, \nu_1-2, ..., \nu_k-2\},\end{equation} with
$\nu_1, ..., \nu_k$ the possible growth rates in $(1, 2)$ of pluriharmonic functions
on $C$. Then for all $c> 0$, there exists an {\rm AC} Calabi-Yau metric $g_c$ on $M$ whose associated K\"ahler form $\omega_c$ lies in $\mathfrak{k}$ and
\begin{equation}
\Phi^*\omega_c - c\omega_0 = O(r^{\max\{-2n,\,\nu\}})
\end{equation}
with $g_0$-derivatives.
Moreover, if $\nu < -2n$, then there exists $\epsilon > 0$ such that
\begin{equation}\Phi^*\omega_c - c\omega_0 = const \cdot i\partial\bar{\partial}r^{2-2n} + O(r^{-2n-1-\epsilon}).\end{equation}
\end{theorem}

\begin{remark}
If $\omega - \xi = d\eta$ globally on $M$ in Definition \ref{d:acsupp}, then we only need Part (i) of Corollary \ref{c:ideldelbar} for the proof of Theorem \ref{thm:main} rather than the more difficult Part (ii). In particular, the proof then works for $n = 2$ as well. This is the case in our applications to quasiprojective manifolds \cite{ch2}.
\end{remark}

\begin{remark}\label{r:crem}
The parameter $c$ corresponds to scaling if $(M,J)$ is Stein and $\mathfrak{k} = 0$, and to the flow of a holomorphic vector field if $(M,J)$ is a crepant resolution of the cone $(C,J_0)$; cf.~Corollary \ref{c:whatisc}. However, we do not know whether or not $c$ is caused by scaling and diffeomorphism in general.
\end{remark}

\subsection{Linear analysis}\label{s:linear_analysis}
We require a definition of weighted H\"older spaces.

\begin{definition}
Let $(M, g)$ be AC with tangent cone $(C,g_{0})$, and let $\rho$ be a radius function.

(i) For $\beta\in\R$ and $k$ a nonnegative integer, define $C_{\beta}^{k}(M)$ to be the space of continuous functions $u$ on $M$ with $k$ continuous derivatives such that $$\norm{u}_{C^{k}_{\beta}} :=\sum_{j=0}^{k}\sup_{M}|\rho^{j-\beta}\nabla^{j}u| < \infty.$$
Define $C_{\beta}^{\infty}(M)$ to be the intersection of the $C_{\beta}^{k}(M)$ over all $k\in \N_0$.

(ii) Let $\delta(g)$ be the convexity radius of $g$, and write $d(x,y)$ for the distance between two points $x$ and $y$ in $M$. For $T$ a tensor field on $M$ and $\alpha,\gamma\in\R$, define
$$[T]_{C^{0,\alpha}_\gamma} :=\sup
_{\substack{x\,\neq\,y\,\in\,M \\
d(x,y)\,<\,\delta(g)}}\left[\min(\rho(x),\rho(y))^{-\gamma}\frac{|T(x)-T(y)|}{d(x,y)^{\alpha}} \right] ,$$
where $|T(x)-T(y)|$ is defined via parallel transport along the minimal geodesic from $x$ to $y$.

(iii) For $\beta\in\R$, $k$ a nonnegative integer, and $\alpha\in(0,1)$, define the weighted H\"older space $C_{\beta}^{k,\alpha}(M)$ to be the set of $u\in C_{\beta}^{k}(M)$ for which the norm $$\norm{u}_{C_{\beta}^{k,\alpha}}:=\norm{u}_{C^{k}_{\beta}} +[\nabla^{k}u]_{C^{0,\alpha}_{\beta-k-\alpha}} < \infty.$$
\end{definition}

\begin{remark}
Whether one decides to measure the asymptotics of a function $u\in C_{\beta}^{k}(M)$ in terms of the metric $g$ or $g_{0}$ actually makes no difference.
\end{remark}

Let $(C,g_{0})$ be a Riemannian cone with $\dim_\R C = m > 2$ {and} with link $L$, {and} let $(M^m,g)$ be AC with tangent cone $(C,g_{0})$. Consider the Laplacian $\Delta$ on functions derived from $g$. Then
\begin{equation}\label{laplacian}
\Delta:C^{k+2,\alpha}_{\beta+2}(M)\to C_{\beta}^{k,\alpha}(M).
\end{equation}
The main point of defining weighted H{\"o}lder spaces is that there exists a satisfactory elliptic theory associated with the mapping (\ref{laplacian}), which we now summarize. Our reference is Marshall \cite{Marshall}; a major earlier paper by Lockhart-McOwen \cite{Lockhart} develops the theory in weighted Sobolev spaces.

\begin{theorem}[{\cite[Theorem 6.10]{Marshall}}]\label{exceptional}
Define
the set of exceptional weights \begin{equation}\label{hothead}
\mathcal{P}:=\left\{-\frac{m-2}{2} \pm \sqrt{\frac{(m-2)^2}{4} +\mu}
:\mu\geq 0\;\textrm{is an eigenvalue of $\Delta_{L}$}\right\}.
\end{equation}
Then the operator \eqref{laplacian} is Fredholm if and only if $\beta+2\not\in\mathcal{P}$.
\end{theorem}

The elements of $\mathcal{P} \cap \R^+$ are precisely the possible growth rates of polynomial harmonic functions on $C$ and on $M$. For instance,
if $C = \R^m$, then $\mathcal{P} = \N_0 \cup (2-m - \N_0)$. In general, the elements of $\mathcal{P}$ are spread symmetrically about the point $\frac{2-m}{2}$ and lie outside the interval $(2-m,0)$, with $2-m$ and $0$ being exceptional. The weight $2-m$ is associated with the Green's function on $M$.

\begin{remark}\label{r:ricci_weights}
By the Lichnerowicz-Obata theorem, if $(C,g_0)$ has nonnegative Ricci curvature and is not isometric to $\R^m$, then $\mathcal{P} \cap [1-m,1] = \{2-m,0\}$. Moreover, if $(C,g_0)$ is in addition K{\"a}hler, then all of the exceptional weights in the interval $(1,2)$ are associated with pluriharmonic functions on $C$; see Corollary \ref{c:ph}.
\end{remark}

The main point for us then is the following existence result:

\begin{theorem}[{\cite[Theorem 6.10]{Marshall}}]\label{surjectivelaplacian}
If $\beta+2 \in (2-m, \infty)\setminus\mathcal{P}$, then the map \eqref{laplacian} is surjective.
\end{theorem}

Using Theorem \ref{exceptional}, this is essentially dual to the fact that (\ref{laplacian}) is {injective} for
$\beta < -2$, which is in turn clear from the maximum principle. If $f \in C^{k,\alpha}_\beta(M)$ with $\beta < -m$, then
$\Delta u = f$ is of course
still solvable, but the leading term of $u$ will typically be a constant multiple of $r^{2-m}$.

\subsection{Existence if the right-hand side decays fast}\label{s:fast_decay}

We will now explain how Parts (i) and (ii) of Theorem \ref{t:existence} follow by using well-known methods. See \cite{vanC3} for a more detailed write-up.

If $M$ is ALE, i.e.~if the tangent cone is flat, both parts were proved by Joyce \cite[Theorem 8.5.1]{Joyce}. Joyce's strategy for Part (ii) carries over to the general AC case. We will thus assume for now that $\beta \in (-2n,-2)$, and explain the extra arguments needed to deal with Part (i) later. We make use of the continuity method to solve the equations
$(\omega+i\p\bar{\p}u_{t})^{n} =e^{tf}\omega^{n}$ in $C^{k,\alpha}_{\beta + 2}(M)$ for all $k$. Openness follows from Theorem \ref{surjectivelaplacian} because $\beta \in (-2n,-2)$; closedness follows from a priori estimates.

The key to these estimates is to establish the following Sobolev inequality: For each AC manifold $(M,g)$ of real dimension $m > 2$, there exists a constant $C<\infty$ such that
\begin{equation}\label{sobolev}
\left(\int_{M}|u|^{\frac{2m}{m-2}}\,d\vol\right)^{\frac{m-2}{m}}\leq C\int_{M}|\nabla u|^{2}\, d\vol
\end{equation}
for all $u \in C^\infty_0(M)$. Once this estimate is known, Moser iteration with weights yields a $C^0_{\beta + 2}$ bound on $u_t$.
Yau's method then provides
a uniform bound on $i\partial\bar{\partial}u_t$, hence on $[i\partial\bar{\partial}u_t]_{C^{0,\alpha}}$ by Evans-Krylov,
which together with the $C^0_{\beta + 2}$ bound from before can be bootstrapped into a $C^\infty_{\beta + 2}$ bound.

In the ALE case, Joyce notes that (\ref{sobolev}) can be deduced from elliptic theory in weighted Sobolev spaces; details can be found in Pacini's work \cite[\S 13]{Pac}, which in fact deals with the general AC case. Tian and Yau \cite[\S 3]{Tian} proved (\ref{sobolev}) in a number of special cases by
using the solution of the Plateau problem and the Michael-Simon Sobolev inequality. General and completely elementary proofs were given by van Coevering \cite[\S2.2]{vanC3} and by the second author \cite[Theorem 1.2]{Hein2}.

We can now explain the proof of Part (i). If
$\beta < -2n$, then Part (ii) tells us that there exists a unique solution $u$ $=$ $O_\epsilon(r^{2-2n+\epsilon})$. Thus, by expanding the Monge-Amp{\`e}re equation, $\Delta_{g_0} u = O(r^\beta)$ on $M \setminus K$ if $\beta > \max\{-4n, -2n+\lambda\}$, where $\lambda < 0$ is the rate of $M$.
In the ALE case, we can then invoke standard
Green's function estimates on $\C^n$ to deduce that $u = const \cdot r^{2-2n} + O(r^{\beta+2})$ if, in addition, $\beta > -2n-1$.
This is generalized to AC manifolds in \cite{vanC3}, using some abstract heat kernel estimates. Alternatively, one can expand the relation $\Delta_{g_0} u = O(r^\beta)$ according to eigenfunctions of
the Laplacian on each slice of the cone, and then analyze the resulting radial ODE's.

\subsection{Existence if the right-hand side decays slowly}\label{s:slow_decay}

The purpose of this section is to explain how Part (iii) of Theorem \ref{t:existence} can be reduced to Part (ii) by using the following lemma.

\begin{lemma}\label{prop:main}
Suppose that $\beta \in (-2,0)$ and $\beta + 2 \notin \mathcal{P}$. If $f \in C^\infty_\beta(M)$, then there exists a function $u_1 \in C^\infty_{\beta + 2}(M)$ such that $\omega + i\partial\bar{\partial}u_1 > 0$ and $(\omega + i\partial\bar{\partial} u_1)^n = e^{f - f_1}\omega^n$ for some $f_1 \in C^\infty_{2\beta}(M)$.
\end{lemma}

Let us take this for granted for now. Then $\omega_1 := \omega + i\partial\bar{\partial}u_1$ is an AC K{\"a}hler metric again. If, by chance, $2\beta < -2$, then Part (ii) of Theorem \ref{t:existence} tells us that $(\omega_1 + i\partial\bar{\partial}u_2)^n = e^{f_1}\omega_1^n$ is solvable with $u_2 \in C^\infty_{2\beta + 2}(M)$, and $u := u_1 + u_2$ will be a solution to the original equation. On the other hand, if we still have $2\beta \geq -2$, then Lemma \ref{prop:main} allows us to construct a function $u_2 \in C^\infty_{2\beta+2+\epsilon}(M)$ for all
$\epsilon > 0$ such that $(\omega_1 + i\partial\bar{\partial}u_2)^n = e^{f_1 - f_2}\omega_1^n$, where $f_2 \in C^\infty_{4\beta + \epsilon}(M)$ for all $\epsilon > 0$. It is then clear that
we can proceed iteratively until the rate of the error drops below $-2$.

\begin{proof}[Proof of Lemma \ref{prop:main}]
We identify $M\setminus K$ and $(1,\infty) \times L$ via $\Phi$.
Let $\eta:\mathbb{R}^+\rightarrow\mathbb{R}^+$ be smooth with
$\eta(t)=0$ for $t\leq 1$ and $\eta(t)=1$ for $t \geq 2$. Given $R > 1$, define $\eta_{R}:M\to\R$ by
$\eta_{R} :=\eta \circ (r/R)$. Since $\Delta:C_{\beta+2}^{\infty}(M)\to
C_{\beta}^{\infty}(M)$ is surjective {by} Theorem \ref{surjectivelaplacian}, we can find a function $\hat{u}_1\in C_{\beta+2}^{\infty}(M)$
solving the equation $\Delta\hat{u}_1=2f$
on $M$ (it would be enough to solve this away from a large
compact set).
We now claim that $u_1 := \eta_R \hat{u}_1$ has all the desired properties if $R \gg 1$.

As for positivity of the closed $(1,1)$-form $\omega + i\partial\bar{\partial}u_1$, consider that
$$
i\partial\bar{\partial}u_1 =\eta_{R}i\partial\bar{\partial}\hat{u}_1
+ i\frac{\eta'}{R}\left(\partial\hat{u}_1 \wedge \bar{\partial}r + \partial r \wedge \bar{\partial}\hat{u}_1\right) + i\hat{u}_1\left(\frac{\eta'}{R} \partial\bar{\partial}r + \frac{\eta''}{R^2}\partial r \wedge \bar{\partial}r\right).
$$
The length of this form is $O(R^{\beta})$ because $\hat{u}_1 \in C^\infty_{\beta + 2}(M)$, $r \in C^\infty_{1}(M\setminus K)$, and $\eta' \circ (r/R)$ vanishes unless $r \in [R,2R]$.
Thus,
$\sup |i\partial\bar{\partial}u_1| \to 0$ as $R \to \infty$, and so $\omega + i\partial\bar{\partial}u_1 > 0$ if $R \gg 1$, as desired.

As for the volume form condition, observe that, for $r > 2R$, we have
\begin{align*}
(\omega + i\partial\bar{\partial}u_1)^n &= (1 + \frac{1}{2}\Delta \hat{u}_1 )\omega^n + {n \choose 2} (i\partial\bar{\partial}\hat{u}_1)^2 \wedge \omega^{n-2} + ... + (i\partial\bar{\partial}\hat{u}_1)^n\\
&= (1 + f + O(r^{2\beta})) \omega^n,
\end{align*}
so that the lemma follows with $f_1 = f - \log(1 + f + O(r^{2\beta})) = O(r^{2\beta})$ at infinity.
\end{proof}

\begin{remark}
Section 2 of \cite{Tian} is devoted to a similar \textquotedblleft preconditioning\textquotedblright, in which a certain initial AC K{\"a}hler metric gets modified step by step to improve the decay of its Ricci potential. We have found this part of \cite{Tian} difficult to follow. Our proof here exhibits a clean way of doing the iterative improvement if one is willing to invest some basic elliptic theory in weighted H{\"o}lder spaces.
\end{remark}

\subsection{Application to Calabi-Yau metrics}\label{s:cy_ex_proof}
This section is dedicated to the proof of Theorem \ref{thm:main}. We begin with two preliminary lemmas.

\begin{lemma}\label{complexasym}
In the setting of Theorem \ref{thm:main}, we have
$\Phi^*J-J_{0} =O(r^{\lambda})$ with $g_0$-derivatives.
\end{lemma}

\begin{proof}
We again identify $M \setminus K$ and $(1,\infty) \times L$ via $\Phi$, and we tacitly allow ourselves to work away from increasingly large compact sets whenever necessary. All metric quantities and operations will be the ones associated with $g_0$. We only prove $C^0$ decay, using linear algebra as in \cite[\S2]{donaldson:96}.

Let us begin by defining the complex vector spaces $$U_0 := \wedge_{J_{0}}^{1,\,0}(T_{x}^{\C}M)^{*}, \;\; U :=  \wedge_{J}^{1,\,0}(T_{x}^{\C}M)^{*},\;\;V:=(T^{\C}_{x}M)^*,$$ for each $x \in M \setminus K$. Then $U_0, U \subset V$ with natural Hermitian structures induced by $g_0$. \medskip\

\noindent \emph{Claim:} There exists a $\C$-linear map $\mu: U_0 \to \bar{U}_0$,
$\norm{\mu} \leq C|\Omega-\Omega_{0}|$, with $U = \{u + \mu(u): u \in U_0\}$.\medskip\

Let us first see how this implies $C^0$ decay for $J$.
Any $\sigma \in U_0$ can obviously be written as
$$\sigma =[u+\mu(u)]+[v+\bar{\mu}(v)],\;\; u:=(1-\bar{\mu}\mu)^{-1}(\sigma), \;\;  v :=-\mu(u),$$
bearing in mind {the fact} that $\|\mu\| \ll 1$, so that $1-\bar{\mu}\mu$ is indeed invertible. As a consequence,
\begin{equation*}
|(J-J_{0})\sigma|= |i(u+\mu(u)) -i(v+\bar{\mu}(v)) -i\sigma| \leq C|\Omega-\Omega_0| |\sigma|,
\end{equation*}
which implies what we need to know.\medskip\

\noindent \emph{Proof of the claim.} Set $W:=\wedge^{n+1}V$ and define maps $T, T_0: V\to W$ by  $T_{(0)}\alpha=\Omega_{(0)}\wedge\alpha$.
Then
$$\ker T_0 = U_0, \;\; \ker T = U, \;\; (\ker T_{0})^{\perp}=\bar{U}_0.$$
Let $\pi_0$ denote the $g_0$-orthogonal projection from $W$ onto $\im T_0$, and define $T':= \pi_0 \circ T$. Obviously $\norm{T-T_{0}} \leq C|\Omega-\Omega_{0}| \ll 1$ and $\dim \ker T_0 = \dim \ker T$,
so that $\pi_0$ restricts to an isomorphism from $\im T$ onto $\im T_0$, and $T'$ restricts to an isomorphism from $\bar{U}_0$ onto $\im T_0$. It is now clear that $$\mu:=-(T'|_{\bar{U}_0})^{-1}\circ (T'|_{U_0})$$ presents $U$ as a graph over $U_0$. Moreover,
\begin{align*}
\norm{T'|_{U_0}} &=\norm{(T'-T_{0})|_{U_0}} \leq C|\Omega-\Omega_{0}|,\\
\norm{(T'|_{\bar{U}_0})^{-1}}
&\leq \norm{(T_{0}|_{\bar{U}_0})^{-1}} +
C|\Omega-\Omega_{0}| \leq C,
\end{align*}
and so $\|\mu\|$ does indeed decay in the desired fashion.
\end{proof}

Given Lemma \ref{complexasym} and Theorem \ref{t:existence}, we can now follow \cite[\S 4.2]{vanC}. The only difference is that we make use of the fact that $r^{2\alpha}$ is strictly plurisubharmonic for all $\alpha > 0$, not just $\alpha = 1$.

\begin{lemma}\label{1convex}
Let $(M, g, J)$ be a K\"ahler manifold and {let} $(C, g_0, J_0)$ {be} a K{\"a}hler cone with radius $r$ such that there exist $\lambda < 0$, a compact $K\subset M$, and a diffeomorphism $\Phi:\{r > 1\}\to M\setminus K$, with $|\nabla_{0}^{k}(\Phi^{*}J-J_{0})|_{{0}}=O(r^{\lambda-k})$ for $k=0,1$. Then, for all $\alpha > 0$, $M$ admits a smooth plurisubharmonic function $h_\alpha$ which is strictly plurisubharmonic and equal to $(r \circ \Phi^{-1})^{2\alpha}$ outside a compact $K_\alpha$.
\end{lemma}

\begin{proof}
We identify $M \setminus K$ and $\{r > 1\}$ via $\Phi$. Let $\psi:\R^+\to\R^+$ be smooth with $\psi',\psi''\geq 0$ and
$$
\psi(t) = \begin{cases}
T+2 & \textrm{if}\;\,t<T + 1,\\
t & \textrm{if}\;\,t>T+3,
\end{cases}
$$
for some $T = T_\alpha > 1$ to be specified later. Then $h_\alpha :=\psi \circ r^{2\alpha}: M \to \R^+$ satisfies
$$
i\partial\bar{\partial}h_\alpha =
\begin{cases}
0 &\textrm{on}\;K \cup \{1 < r < (T+1)^{\frac{1}{2\alpha}}\},\\
\psi'' i\p r^{2\alpha}\wedge\bar{\p}r^{2\alpha}+\psi'i\p\bar{\p}r^{2\alpha} &\textrm{on}\;\{r > T^{\frac{1}{2\alpha}}\}.
\end{cases}
$$
Since $i\p u \wedge\bar{\p}u \geq 0$ with respect to $J$ for any smooth real-valued function $u$, it suffices to prove that $i\partial \bar{\partial}r^{2\alpha} > 0$ with respect to $J$ on $\{r > R\}$, provided that $R \gg 1$. Indeed, on $M \setminus K$,
\begin{equation*}
\begin{split}
|i\p\bar{\p}r^{2\alpha}-i\p_{{0}}\bar{\p}_{{0}}r^{2\alpha}|_{{0}}=
\frac{1}{2}|d((J-J_{0})dr^{2\alpha})|_{{0}} = O(r^{2\alpha - 2 + \lambda}).
\end{split}
\end{equation*}
This is of lower order compared to $i\p_{{0}}\bar{\p}_{{0}}r^{2\alpha}$, which one easily checks is positive with respect to $J_0$.
Using $J - J_0 = O(r^\lambda)$ again, this shows that $i\p\bar{\p}r^{2\alpha}$ is positive with respect to $J$ for $R \gg 1$.
\end{proof}

\begin{proof}[Proof of Theorem \ref{thm:main}]
We identify $M\setminus K$ and $(1,\infty) \times L$ via $\Phi$ and allow ourselves to work off of increasingly large compact sets if convenient.
By Lemma \ref{complexasym}, $J-J_{0}=O(r^{\lambda})$ with $g_0$-derivatives. {As a consequence}, by Lemma \ref{1convex}, $M$ admits smooth plurisubharmonic exhaustions $h_\alpha$ ($\alpha > 0$) that are strictly plurisubharmonic and equal to $r^{2\alpha}$ outside a compact set. Moreover, for all $k \in \N_0$,
\begin{equation}\label{asym}
|\nabla^k_0(i\partial\bar{\partial}r^{2}-\omega_{0})|_{g_0}=O(r^{\lambda-k}).
\end{equation}

By assumption, we have a K{\"a}hler form $\omega$ on $M$ and a smooth real $(1,1)$-form $\xi$ on $M \setminus K$ with $|\nabla_0^k\xi|_0 = O(r^{\mu-k})$ such that
$\omega -\xi$ is $d$-exact on $M \setminus K$. Thus, by Corollary \ref{c:ideldelbar}(ii), $\omega - \xi = -i\p\bar{\p}u$ on $\{r > R\}$ for some $R > 1$ and some smooth real{-valued} function $u$. Fix $\alpha \in (0,1)$. We can assume that
$h_\alpha = r^{2\alpha}$ as well as $h_1 = r^2$ on $\{r > R\}$, and that both functions are strictly plurisubharmonic on this region.
{Also}, fix a cutoff function $\zeta:M\to\R$ with
$$
\zeta(x) = \begin{cases}
0 & \textrm{if}\;\,r(x)< 2R,\\
1 & \textrm{if}\;\,r(x)>3R,
\end{cases}
$$
and define $\zeta_{S}(x) := \zeta(x/S)$ in the obvious {way} for $S > 2$.
Given $c > 0$, we now construct
$$\hat{\omega} :=\omega+i\p\bar{\p}(\zeta u) + C i\p\bar{\p} ((1-\zeta_S)h_\alpha) + ci\p\bar{\p}h_1,$$
with $C$ and $S$ to be determined.
{Note that} $\hat{\omega} = \omega+C i\p\bar{\p} h_\alpha + ci\p\bar{\p}h_1 \geq \omega > 0$ on $K \cup \{1 < r < 2R\}$ because $h_\alpha$ and $h_1$ are plurisubharmonic;
$\hat{\omega} =\xi + C i\p\bar{\p} h_\alpha + ci\p\bar{\p}h_1 > 0$ on $\{3R < r < 2SR\}$, after increasing $R$ if necessary, because $|\xi| = O(r^\mu)$; $\hat{\omega} > 0$ on $\{3SR < r\}$ for the same reason; $\hat{\omega} > 0$ on
$\{2R \leq r \leq 3R\}$ by compactness if $R$ is now fixed and $C$ is made large enough; and finally,
$\hat{\omega} > 0$ on $\{2SR \leq r \leq 3SR\}$ if $S \gg 1$ depending on all previous choices because $h_\alpha$ is of lower order compared to $h_1$.
In conclusion, $\hat{\omega}$ is a genuine K\"ahler form {on $M$} with $\hat{\omega} = \xi + ci\p\bar{\p}r^{2}$ at infinity.

Let us denote the metric corresponding to $\hat{\omega}$ by $\hat{g}$. Then, from \eqref{asym}, for all $k \in \N_0$,
$$
|\nabla_{0}^{k}(\hat{g}-cg_{0})|_{{0}}=O(r^{\nu-k}).
$$
As a result, $(M,\hat{g},J)$ is an AC K{\"a}hler manifold {of} rate $\nu$, and $\hat{g}$ has a global Ricci potential
$$
\hat{f} =\log\left(\frac{i^{n^2}\Omega\wedge\bar{\Omega}}{(\hat{\omega}/c)^{n}}\right) \in C^\infty_{\nu}(M).
$$
We may now appeal to Theorem \ref{t:existence} in order to solve the complex Monge-Amp{\`e}re equation $$(\hat{\omega} + i\partial\bar{\partial}\hat{u})^n = e^{\hat{f}}\hat{\omega}^n$$
for $\hat{u}$, bearing in mind that, by
Remark \ref{r:ricci_weights}, $\mathcal{P} \cap (0,2) = \{\nu_1,..., \nu_k\}$.
\end{proof}
\newpage
\section{Uniqueness in a given K{\"a}hler class}\label{s:uniqueness}

Let $(M^n,J)$ be an open complex manifold with a nowhere vanishing holomorphic volume form $\Omega$. Let $\omega_1,\omega_2$ be Ricci-flat K\"ahler forms in the same de Rham class on $M$, with volume form $i^{n^2}\Omega \wedge \bar{\Omega}$,
{whose} associated {K\"ahler} metrics $g_1, g_2$ are complete. 
In this generality, $g_1, g_2$ need not even have the same volume growth rate. However, if one assumes in addition that $g_1$ and $g_2$ are asymptotically isometric at infinity, then one might perhaps expect that $g_1 = \Phi^*g_2$ for some $\Phi \in {\rm Aut}(M,J,\Omega)$, or at least 
that $g_1 = \Phi^*g_2$ for \emph{some} reasonably canonical map $\Phi \in {\rm Diff}(M)$.
One difficulty here stems from the fact that $(M,J)$ may have a huge automorphism group. For example, any map of the form
$(z,w) \mapsto (z, w + f(z))$ with $f \in \mathcal{O}(\C)$ is an automorphism of $\C^2$ that preserves $dz \wedge dw$.

What one \emph{can} prove in some cases is that $\omega_1 = \omega_2$ if $\omega_1 - \omega_2$ has rapid decay. In the ALE case, this is due to Joyce \cite[Theorem 8.4.4]{Joyce} under the assumption that $\omega_1 - \omega_2 = O(r^{-n-\epsilon})$. {This rate} is good enough for applications because ALE Calabi-Yau
metrics decay like $r^{-2n}$.
Direct extensions to the AC case were proved in \cite{goto, vanC3}, but $O(r^{-n-\epsilon})$ now no longer suffices for many applications. 
We pursue an entirely different
approach here that allows us to relax $O(r^{-n-\epsilon})$ to $O(r^{-\epsilon})$.

\begin{theorem}\label{t:uniqueness}
Let $(M, J)$ be an $n$-dimensional complex manifold with {\rm AC} K{\"a}hler metrics $g_1,g_2$. If $\eta = \omega_1 - \omega_2$ is $d$-exact, $\eta \in C^\infty_{-\epsilon}(M)$ for some $\epsilon > 0$, $\omega_1^n = \omega_2^n$, and ${\rm Ric}(g_i) \geq 0$, then $\eta = 0$.
\end{theorem}

\begin{remark}
Pulling back the flat metric on $\C^n$ by a $1$-parameter subgroup of SL$(n,\C)$ shows that $\eta = O(r^{-\epsilon})$ cannot be relaxed to $\eta = O(1)$. Theorem \ref{t:uniqueness} is also false for volume growth $2n-1$ in place of $2n$; the Taub-NUT metric will serve as a useful counterexample throughout the proof.
\end{remark}

\begin{remark}
Carron \cite{Carron} proved a Calabi-Yau type uniqueness theorem, assuming an intermediate rate of decay, on the Hilbert scheme of $N$ points in $\C^2$ with a QALE hyper-K{\"a}hler metric. It would 
be interesting to see whether there exists a common extension of Carron's result and ours.
\end{remark}

The rest of this section is dedicated to proving Theorem \ref{t:uniqueness}. Some conventions first:

$\bullet$ A $1$-form $\alpha$ is \emph{harmonic} if $-\Delta\alpha = (dd^* + d^*d)\alpha = 0$.

$\bullet$ A tensor $T$ on a Riemannian cone is \emph{homogeneous} if $L_{r\partial_r}T = \mu T$ for some $\mu \in \R$. 

$\bullet$ The \emph{rate} of $T$ is the infimum of all $\lambda$ such that $|\nabla^k T| = O(r^{\lambda - k})$ as $r \to \infty$ for every $k$.

$\bullet$ $\epsilon_0 \in (0,1)$ will denote a small constant that only depends on the manifold in question.

Also, for context, recall the following version of the Lichnerowicz-Obata theorem:~If a harmonic function on a Riemannian cone with ${\rm Ric} \geq 0$ has positive rate, then the rate is in fact at least one, and equality is attained precisely for the linear functions on Euclidean space.

 The following lemma, due to Cheeger and Tian \cite{Cheeger}, contains the key idea. A priori there would not even be a reason to expect that $\alpha$ is closed; yet the conclusion is that $\alpha$ must be exact.

\begin{lemma}[{\cite[Lemma 7.27]{Cheeger}}]\label{l:slowgrowth1form}
Let $\alpha$ be a homogeneous harmonic $1$-form of rate $\lambda$ on a cone with ${\rm Ric} \geq 0$. If $\lambda \in [0,1)$, then $\alpha = du$ for a homogeneous harmonic function $u$ of rate $\lambda + 1$.
\end{lemma}

\begin{remark}\label{r:hmg1forms}
As for harmonic functions, there are only very few possibilities for harmonic $1$-forms $\alpha$ on a cone $C$ \emph{not} to be sums of homogeneous ones, or in other words, to contain any $\log$ terms at all. Working through the relevant separation of variables based on \cite[(2.14), (2.15)]{Cheeger}, one finds that the only possible sources of $\log$ terms are $\dim C = 3$
and $\alpha = r^{-(1/2)}(\log r)(\kappa\, dr + 2r\, d\kappa)$ with $\kappa$ a $\frac{3}{4}$-eigenfunction on the link $L$, or $\dim C= 4$ and $\alpha = \frac{1}{r}(\log r)\eta$ with $\eta$ the radially parallel extension of a harmonic $1$-form on $L$. Neither of these occurs if ${\rm Ric} \geq 0$. See also \cite[p.~546]{Cheeger}.
\end{remark}

By applying Lemma \ref{l:slowgrowth1form}
to $\alpha = d^cv$ with $v$ harmonic, we can draw the following simple corollary which, to our knowledge, has not been observed before.

\begin{corollary}\label{c:ph}
Any homogeneous harmonic function of rate in $[1,2)$ on a K\"ahler cone with ${\rm Ric} \geq 0$ must  already be pluriharmonic.
\end{corollary}

We now wish to lift Lemma \ref{l:slowgrowth1form} from cones to AC manifolds. This requires some preparations.

\begin{lemma}\label{l:transplant}
Let $(M^m,g)$ be an {\rm AC} Riemannian manifold $(m > 2)$ with tangent cone $(C,g_0)$.

{\rm (i)} Let $\alpha$ be a harmonic $1$-form of rate $\lambda$ on $M$.
Then $\alpha = \alpha_0 + \beta$ outside a compact set, where $\alpha_0$ is a harmonic $1$-form of rate $\lambda$ on $C$ and  $\beta$ has rate at most $\lambda-\epsilon_0$.

{\rm (ii)} Let $u_0$ be a harmonic function of rate $\lambda  > 0$ on $C$. Then there exists a harmonic function $u$ of rate $\lambda$ on $M$ such that $u = u_0 + v$ outside a compact set, with $v$ of rate at most $\lambda-\epsilon_0$.
\end{lemma}

\begin{proof}
(i) It suffices to solve the equation $\Delta_{g_0}\beta = -\Delta_{g_0}\alpha = O(r^{\lambda-2-\epsilon_0})$ for $\beta$ of rate at most $\lambda - \epsilon_0$, defined outside a large compact set.
There are no obstructions to doing this: Using (2.14), (2.15) of \cite{Cheeger} and the spectral decomposition of the $1$-form Laplacian on $L$, one can split up $\Delta_{g_0}\beta = -\Delta_{g_0}\alpha$ as an infinite sequence of Euler-Cauchy type ODE's.
These are easy to solve individually, with the correct behavior at infinity.
Summing the solutions is not difficult either because $-\Delta_{g_0}\alpha$ is smooth, 
so that its  Fourier coefficients on each slice decay rapidly in terms of the spectral parameter.

(ii) Extend $u_0$ to a function $\bar{u}_0$ on $M$. Then $\Delta_g \bar{u}_0 \in C^\infty_{\lambda-2-\epsilon_0}(M)$. Since we have $\lambda - 2 - \epsilon_0 > -m$, Theorem \ref{surjectivelaplacian} tells us that there exists $v$ of rate at most $\lambda-\epsilon_0$ such that $\Delta_g v = -\Delta_g\bar{u}_0$. (This type of argument is quite common; see e.g.~\cite{cz, donn} for very similar results and applications.)
\end{proof}

The following is then the promised extension of Lemma \ref{l:slowgrowth1form} from cones to AC manifolds. Notice that the Bochner formula already tells us that $M$ does not admit any nonzero harmonic $1$-forms of negative rate; this is in fact an important ingredient for the proof of the theorem.

\begin{theorem}\label{t:slowgrowth1form}
Let $M$ be an {\rm AC} Riemannian manifold with ${\rm Ric} \geq 0$. Let $\alpha$ be a harmonic $1$-form of rate $\lambda$ on $M$.
If $\lambda \in [0,1)$, then $\alpha = du$ for some harmonic function $u$ of rate $\lambda + 1$.
\end{theorem}

\begin{proof}
By Lemma \ref{l:transplant}(i), $\alpha = \alpha_0 + \beta$ outside a compact set, where $\alpha_0$ is harmonic of rate $\lambda$ on the cone and $\beta$ has rate at most $\lambda - \epsilon_0$. Since harmonic $1$-forms on $C$ are sums of homogeneous ones by Remark \ref{r:hmg1forms}, we can assume that $\alpha_0$ is in fact homogeneous. Then, by Lemma \ref{l:slowgrowth1form}, 
${\alpha}_0 = du_0$ for a harmonic function $u_0$ of rate $\lambda + 1$ on $C$. By Lemma \ref{l:transplant}(ii), there exists a harmonic function $u$ of rate $\lambda + 1$
on $M$ such that $u = u_0 + v$ away from a compact set, with $v$ of rate at most $\lambda + 1 - \epsilon_0$. Thus,
$\alpha = du + \hat{\alpha}$ on $M$, where $\hat{\alpha}$ has rate at most $\lambda - \epsilon_0$ and is necessarily harmonic. 

We now repeat the argument with $\hat{\alpha}$ in place of $\alpha$ for as long as the rate of the remainder stays nonnegative. Since the rate of the remainder drops by at least some universal $\epsilon_0$ at each stage, we can assume that $\alpha = du + \hat{\alpha}$ after finitely many iterations, where 
$u$ is harmonic of rate $\lambda + 1$ and the harmonic $1$-form $\hat{\alpha}$ has a negative rate.
On the other hand, since ${\rm Ric} \geq 0$, the Bochner formula tells us that $\Delta|\hat{\alpha}|^2 \geq 0$ globally on $M$. Thus, $\hat{\alpha} = 0$ by the strong maximum principle. 
\end{proof}

We again have a direct consequence, which we need for the proof of Theorem \ref{t:uniqueness}.

\begin{corollary}\label{c:subquadratic}
Any harmonic function of rate strictly less than $2$ on an {\rm AC} K{\"a}hler manifold with ${\rm Ric} \geq 0$ must already be pluriharmonic.
\end{corollary}

\begin{remark}\label{r:liremark}
Corollary 5 in Li \cite{li} states that $o(r^2)$ harmonic functions on \emph{every} complete K{\"a}hler manifold 
with ${\rm Ric} \geq 0$ are pluriharmonic. 
This is false:~Consider the Taub-NUT manifold, which is complete hyper-K{\"a}hler of real dimension $4$ with cubic volume growth, with a triholomorphic Killing field $X$ that rotates the circles at infinity \cite{mintaubnut}. Fix a parallel complex structure $J$ and define $u$ by $du$ $=$ $X\,\llcorner\,\omega$. Then $u$ is harmonic of linear growth (asymptotic to the linear function on $\R^3$ determined  by $J \in S^2 \subset \R^3$), but not $J$-pluriharmonic because $i\partial\bar{\partial} u = \nabla X$ is a nonzero $L^2$-harmonic $2$-form.

The proof of Li's theorem relies on an incorrect Bochner formula \cite[(14)]{li}, whose correct version \cite[Lemma 4.1]{liwang} involves the full curvature tensor. As a consequence, Li's theorem does hold under the assumption of  nonnegative bisectional rather than Ricci curvature.
\end{remark}

A second consequence of Theorem \ref{t:slowgrowth1form} which is important for us is the following $i\partial\bar{\partial}$-lemma.

\begin{theorem}\label{t:iddbar}
Let $M$ be an {\rm AC} K{\"a}hler manifold with ${\rm Ric} \geq 0$. Let $\eta$ be a $d$-exact real $(1,1)$-form on $M$ with $\eta \in C^\infty_{-\epsilon}(M)$ for some $ \epsilon \in (0,\epsilon_0)$. Then $\eta = i\partial\bar{\partial}u$ with $u \in C^\infty_{2-\epsilon}(M)$.
\end{theorem}

\begin{proof} (1) We first wish to prove that $\eta = d\zeta$, where $\zeta \in C^\infty_{1-\epsilon}(M)$. By assumption,
$\eta = d\nu$ for some $1$-form $\nu$ on $M$. Let $K$ be a large compact set, let $\chi$ be a function with $\chi \equiv 0$ on $K$ and $\chi \equiv 1$ on the complement of some larger compact set, and put $\tilde{\nu} = \chi\nu$ and $\tilde{\eta} = d\tilde{\nu}$. It suffices to construct a $1$-form $\tilde{\zeta} \in C^\infty_{1-\epsilon}(M)$ such that $\tilde{\zeta} \equiv 0$ on $K$ and $d\tilde{\zeta} =\tilde{\eta}$ because then we can set
$\zeta =\nu - \tilde{\nu} + \tilde{\zeta}$.

\text
We write the tangent cone of $M$ as $(0,\infty) \times L$ and indicate the degree of a form by a subscript. Then $\tilde{\eta} = \tilde{\eta}_1 \wedge dr + \tilde{\eta}_2$, where
$\tilde{\eta}_i$ is a $1$-parameter family of $i$-forms on $L$ that vanish identically for 
all $r$ less than some large constant. Similarly,
$\tilde{\nu} = \tilde{\nu}_0 \,dr + \tilde{\nu}_1$, and our assumption that $d\tilde{\nu} = \tilde{\eta}$ can be rewritten as $d_L \tilde{\nu}_0 + \partial_r\tilde{\nu}_1= \tilde{\eta}_1 $ and $d_L \tilde{\nu}_1 = \tilde{\eta}_2$. Consequently, we need to find new solutions $\tilde{\zeta}_0, \tilde{\zeta}_1$ to these equations that still vanish unless $r \gg 1$, but
with growth control at infinity.

The key is to observe that $H^1(L) = 0$ because ${\rm Ric}_L > 0$. Rescaling and applying Hodge theory 
with parameters thus yields a unique 
solution $\tilde{\zeta}_1$ to the second equation, with $d_L^*\tilde{\zeta}_1 = 0$, $\tilde{\zeta}_1 \equiv 0$ unless $r \gg 1$, and $\tilde{\zeta}_1 = O(r^{1-\epsilon})$. 
The first equation is now similar because there must exist $\tilde{\xi}_0$ with $\tilde{\nu}_1 = \tilde{\zeta}_1 + d_L \tilde{\xi}_0$, so that 
 $\tilde{\eta}_1 - \partial_r\tilde{\zeta}_1 = \tilde{\eta}_1 - \partial_r\tilde{\nu}_1  + \partial_r d_L \tilde{\xi}_0 = d_L(\tilde{\nu}_0 + \partial_r\tilde{\xi}_0)$ is indeed $d_L$-exact.

(2) We have $\eta = d\zeta$ with $\zeta \in C^\infty_{1-\epsilon}(M)$ by (1). By Theorem \ref{surjectivelaplacian}, the equation $\frac{1}{2}\Delta u = \bar{\partial}^*\zeta^{0,1}$ 
has a $\C$-valued solution $u \in C^\infty_{2-\epsilon}(M)$. (This is unique up to $\C$-valued pluriharmonics by Corollary \ref{c:subquadratic}, but we don't need to {make use of this fact here}.) Consider the $(0,1)$-form $\xi = \zeta^{0,1} - \bar{\partial} u \in C^\infty_{1-\epsilon}(M)$.
This satisfies $\bar{\partial}\xi = \bar{\partial}^*\xi = 0$. Thus, in particular, $\Delta\xi = 0$.
Theorem \ref{t:slowgrowth1form} now tells us that $\xi = dv = \bar{\partial}v$,
where $v \in C^\infty_{2-\epsilon}(M)$ is $\C$-valued, harmonic, and (obviously) antiholomorphic. This then proves that 
$2{\rm Im}(u + v)$, and in fact $2{\rm Im}(u)$ already, is a potential for $\eta$ with the right asymptotics.
\end{proof}

\begin{remark}
The weighted $i\partial\bar{\partial}$-lemma in Theorem \ref{t:iddbar} can perhaps be viewed as an effective or metric version of the purely complex analytic $i\partial\bar{\partial}$-lemma from Corollary \ref{c:ideldelbar}(i).
\end{remark}

We are now in good shape to prove the main theorem of this section.

\begin{proof}[Proof of Theorem \ref{t:uniqueness}]
We have $\eta = i\partial\bar{\partial}u$ from Theorem \ref{t:iddbar}, where $u \in C^\infty_{2-\epsilon}(M)$. By assumption, 
$(\omega_2 + i\partial\bar{\partial}u)^n = \omega_2^n$. Thus,
$(\Delta_{g_2}u)\omega_2^n = (i\partial\bar{\partial}u)^2 \wedge \psi$, where $\psi$, as well as its derivatives, are uniformly bounded with respect to $g_2$. Theorem \ref{surjectivelaplacian}, together with an obvious iteration argument, now tells us that $u = u' + u''$, where $u'$ is \emph{harmonic} on $M$ of rate at most $2-\epsilon$ and $u''$ has rate at most $2-2n$. By Corollary \ref{c:subquadratic}, $u'$ is in fact pluriharmonic already, so that $(\omega_2 + i\partial\bar{\partial}u'')^n = \omega_2^n$. Since $u''$ goes to zero at infinity, we can now use the strong maximum principle to deduce that $u'' = 0$.
\end{proof}

\begin{remark} 
The following shows how the above proof fails in a slightly modified situation where the statement of Theorem \ref{t:uniqueness} is in fact false. Consider Taub-NUT with a particular choice of complex structure $J$ as in Remark \ref{r:liremark}, and let $u'$ denote the associated harmonic function of linear growth. Then the complex Monge-Amp{\`e}re equation $(\omega + i\partial\bar{\partial}u)^2 =  \omega^2$ has 
a nontrivial solution $u = u' + u''$ with $u'' = O(r^{-1})$ because $i\partial\bar{\partial}u' \neq 0$, but $\omega \wedge i\partial\bar{\partial}u' = 0$. The metric
$\omega + i\partial\bar{\partial}u = \omega + O(r^{-2})$ differs from $\omega$ by flowing along the $J$-holomorphic vector field $\nabla u'$.
\end{remark}

We close this section with a first application of Theorem \ref{t:uniqueness}; see also Remark \ref{r:crem}.

\begin{corollary}\label{c:whatisc}
Consider the $1$-parameter family $\omega_c$ of Ricci-flat metrics in a given K{\"a}hler class $\mathfrak{k}$ constructed in Theorem \ref{thm:main}.

{\rm (i)} If $\mathfrak{k} = 0$, then $\omega_{c_2} =\frac{c_2}{c_1}\omega_{c_1}$ for all $c_1,c_2 > 0$. In this case, $M$ is necessarily Stein.

{\rm (ii)}  At the other extreme, if $M$ is a crepant resolution of $C$, then $\omega_{c_2} =  \exp((\log \frac{c_2}{c_1}) X)^*\omega_{c_1}$ for all $c_1,c_2 > 0$, where $X$ denotes the lift of the holomorphic vector field $r\partial_r$ from $C$ to $M$.
\end{corollary}

\begin{proof} 
The Stein property in (i) follows from the Remmert reduction theorem; see Appendix \ref{s:ddbar}. For (ii), note that the flow of $r\partial_r$ on $C$ lifts to a flow of biholomorphisms of $M$ by the Riemann removable singularities theorem; the pullback action of any such flow preserves every K\"ahler class.
\end{proof}

Neither of the above mechanisms for creating a family of Calabi-Yau metrics in $\mathfrak{k}$ is available in general. There always exists a unique vector field $X_c$ which is $g_c$-\emph{harmonic} and asymptotic to $r \partial_r$; possibly the metrics $g_c$ just differ by scaling and flowing along $X_c$ viewed as a time-dependent vector field. 
We do not know whether or not this is the right picture in general, but see Remark \ref{r:ehrem}.
\newpage
\section{Crepant resolutions}\label{s:crepres}

Let $(C, g_0, J_0,\Omega_0)$ be an $n$-dimensional Calabi-Yau cone. Recall from Theorem \ref{t:affine} that the metric completion $C \cup \{0\}$ can be naturally endowed with the structure of a normal variety $V$ that admits affine embeddings quasihomogeneous with respect to some $\C^*$-action with positive weights.

In many interesting cases, the variety $V$ will admit \emph{crepant resolutions}. By definition, these are surjective holomorphic maps $\pi: M \to V$, where $M$ is a smooth complex manifold with a nowhere vanishing holomorphic volume form $\Omega$ such that $\pi^*\Omega_0 = \Omega$, and $\pi$ is an isomorphism onto its image away from $E = {\rm Exc}(\pi)$. Crepant resolutions are the most obvious class of spaces to try and apply results such as Theorem \ref{thm:main} to in order to construct complete AC Calabi-Yau metrics.

This problem has been studied extensively; we review the main points of the existing theory in Sections \ref{s:classical}--\ref{s:goto} and make some clarifying remarks. Section \ref{s:flag-ex's} presents some new small resolutions associated with flag varieties of compact Lie groups. Our results from Sections \ref{s:existence}--\ref{s:uniqueness} allow for a quick and uniform treatment, though we suspect that an ODE-based approach must exist as well.

\subsection{Classical examples}\label{s:classical}
Every K{\"a}hler cone is naturally the total space of a negative line bundle over some compact K{\"a}hler orbifold with the zero section removed (or blown down). As such, there is always a canonical way of at least partially resolving the cone by pasting the zero section back in, though, in general, this only produces an orbifold. Based on this idea, Calabi \cite{Cal1} constructed 
many examples of AC Ricci-flat K{\"a}hler manifolds which are in fact almost explicit.

\begin{example}[Calabi]\label{ex:calabi1} Let $D$ be a K{\"a}hler-Einstein Fano manifold of complex dimension $n - 1$, $H$ the maximal root of the canonical bundle of $D$ (so that $H^\iota = K_D$ with $\iota$ the Fano index of $D$), and $L \subset H$ the total space of the corresponding U$(1)$-bundle. Then $\pi_1(L) = 0$, and there exist obvious cyclic quotients $L_k = L/\Z_k$ for all $k \in \N$ such that the K{\"a}hler cones $C_k = C(L_k)$ are resolved by the total spaces $M_k$ of the line bundles $H^k$. The cone $C_k$ (equivalently, the manifold $M_k$) admits a global holomorphic volume form if and only if $k\,|\,\iota$, and this then vanishes to order $\frac{\iota}{k}-1$ along the zero section $E_k \cong D$ in $M_k$. In particular, $M_k \to V_k$ is a \emph{crepant} resolution if and only if $k = \iota$.

Calabi first lifts the K{\"a}hler-Einstein metric from $D$ to a Sasaki-Einstein metric on $L$. Solving an
ODE, he then constructs Ricci-flat K{\"a}hler metrics on $M_k$ for all $k \in \N$ that are AC at infinity, but {that have} a cone angle of $2\pi \frac{\iota}{k}$ along the exceptional divisor $E_k$. In particular, his metric is smooth precisely when $k = \iota$. 
Taking $D$ to be projective space, one obtains an ALE space with tangent cone $\C^n/\Z_n$,
recovering Eguchi-Hanson for $n = 2$. Another interesting example is $D = \P^1 \times \P^1$ ($\iota = 2$), in which case $C_1$ is the ordinary double point, or \emph{conifold}, $z_1^2 + z_2^2 + z_3^2 + z_4^2 = 0$ in $\C^4$.

It is worth noting that Calabi's metrics can be universally written as $i\partial\bar{\partial}\sum_{k = 0}^\infty c_{n,k} r^{2-2nk}$ away from a compact set; the $i\partial\bar{\partial} \log r$ term in \cite[(4.14)]{Cal1} cancels with the $\Phi \circ \pi$ term in \cite[(3.1)]{Cal1}.
\end{example}

Let us now consider the ALE case more generally. For $\Gamma$ a {finite} subgroup of SU$(n)$ acting freely on $S^{2n-1}$, $\C^n/\Gamma$ inherits a parallel {holomorphic} volume form from $\C^n$. However, except for $\Gamma = \Z_n$ as above, Calabi's method would produce an AC Ricci-flat \emph{orbifold} that resolves $\C^n/\Gamma$ only partially.
On the other hand, there often exist more complicated resolutions that are actually smooth.

\begin{example}[Kronheimer] The singularities $\C^2/\Gamma$ with $\Gamma < {\rm SU}(2)$ admit crepant resolutions with exceptional set a Dynkin graph of rational curves. 
Kronheimer \cite{Kronheimer} constructed ALE hyper-K{\"a}hler metrics on these resolutions using the McKay correspondence and symplectic reduction.
\end{example}

If we wish to apply Theorem \ref{thm:main} to a crepant resolution $M \to V$, we can take $\lambda = -\infty$, but then the question arises as to which K{\"a}hler classes $\mathfrak{k}$ are in fact $\mu$-almost compactly supported for some $\mu < 0$. If $\mathfrak{k} \in H^2_c(M)$, then $\mu = -\infty$ works and we obtain AC Calabi-Yau metrics with leading
 term $i\partial\bar{\partial}r^{2-2n}$.
This is van Coevering's main result in \cite{vanC3, vanC2} and contains Joyce's foundational work on the ALE case \cite{Joyce}.
Note that the classes represented by Calabi's metrics are compactly supported, and $H^2(M) = H^2_c(M)$ anyway if $M$ is any resolution of $\C^n/\Gamma$. We refer to \cite{Joyce, Sparks2, vanC2, vanC4} for many new examples beyond these, including crepant resolutions of irregular Calabi-Yau cones.

\subsection{Goto's theorem}\label{s:goto} What Section \ref{s:classical} leaves open is whether an \emph{arbitrary} class $\mathfrak{k} \in H^2(M)$ on a
crepant resolution $M \to V$ is $\mu$-almost compactly supported, assuming that $\mathfrak{k}$ 
contains any positive $(1,1)$-forms at all. We will now explain Goto's answer to this question \cite[Theorem 5.1]{goto}.

We begin with the following exact sequence, which holds on every AC K{\"a}hler manifold:
\begin{equation}\label{e:basic_LES}H^1(L) \to H^2_c(M) \to H^2(M) \to H^2(L) \to H^3_c(M).\end{equation}
Here $L$ is the link of $C$, and $H^1(L) = 0$ because ${\rm Ric}_L > 0$. More importantly, the Bochner formula also tells us that $H^2(L) = H^{1,1}_{{\rm pr}, {\rm b}}(L)$, the primitive basic $(1,1)$-cohomology group associated with the Sasaki structure on $L$; {for $L$} the canonical U$(1)$-bundle over a K{\"a}hler-Einstein Fano manifold $D$, this is nothing else but $H^{1,1}_{{\rm pr}}(D)$. {It then follows from Poincar{\'e} duality that}
\begin{equation}\label{e:resol_sequence}
0 \to H^2_c(M) \to H^2(M) \to H^{1,1}_{{\rm pr},{\rm b}}(L) \to H_{2n-3}(E),
\end{equation}
where again $E = {\rm Exc}(\pi)$. 
Thus, for every closed $2$-form $\omega$ on $M$, there exists a compact set $K \subset M$ such that 
$\omega = p^*\xi + d\eta$ on $M \setminus K$, where $p: C \to L$ denotes the radial projection, $\eta$ is a smooth real-valued $1$-form on $M \setminus K$, and $\xi$ is some closed primitive basic $(1, 1)$-form on $L$.

We now observe that $p^*\xi$ is a $(1,1)$-form on $M$ because $M$ and $C$ are biholomorphic at infinity. Since $p^*\xi = O(r^{-2})$, this allows us to invoke Theorem \ref{thm:main} to construct AC Calabi-Yau metrics of rate $-2+\epsilon$ for every $\epsilon > 0$.
Goto did not have the $\beta \in (-2,0)$ case of Theorem \ref{t:existence},
so he needed another observation, which actually yields a better result:
If $\xi$ is a closed primitive basic
$(1,1)$-form on $L$ ($\xi$ is then automatically harmonic), then $p^*\xi$ defines an infinitesimal Ricci-flat deformation of the cone metric $\omega_0$, i.e.~$(\omega_0 + p^*\xi)^n = e^f \omega_0^n$ with $f = O(r^{-4})$ as opposed to merely $O(r^{-2})$.

\begin{theorem}[Goto]\label{crepant}
Let $(C, g_0, J_0,\Omega_0)$ be a Calabi-Yau cone of complex dimension $n > 2$ and let $L$ be the link, $p: C \to L$ the radial projection, and $V$ the
normal affine variety associated {to} $C$. Let $\pi: M \to V$ be a crepant resolution and $\mathfrak{k} \in H^2(M)$ a class that contains positive $(1,1)$-forms. Then for every $c > 0$, there exists a complete Calabi-Yau metric $g_c$ on $M$ such that $\omega_c \in \mathfrak{k}$ and
\begin{equation}\label{e:goto}
\omega_c - \pi^*(c \omega_0) = p^*\xi + O(r^{-4}),
\end{equation} 
where $\xi$ is the primitive basic harmonic $(1,1)$-form on $L$ that represents the restriction of $\mathfrak{k}$ to $L$. Notice that $p^*\xi = O(r^{-2})$. If $\xi = 0$, or equivalently, if $\mathfrak{k} \in H^2_c(M)$, then we even have 
\begin{equation}\label{e:goto2}
\omega_c - \pi^*(c\omega_0) = const \cdot i\partial\bar{\partial} r^{2-2n} + O(r^{-2n-1-\epsilon})
\end{equation}
for some $\epsilon > 0$. This is the special case covered by the earlier works \cite{Joyce, vanC3, vanC2}.
\end{theorem}

\begin{example}[Candelas \& de la Ossa]\label{ex:cdo} One of the very few explicitly known metrics that behave as in (\ref{e:goto}) with $\xi \neq 0$ lives on the so-called \emph{small resolution of the conifold} \cite{delaossa}, that is, on the total space of the rank-$2$ vector bundle $\mathcal{O}_{\P^1}(-1)^{\oplus 2}$.
Contracting the zero section maps this space to the cone $C_1$ over $\P^1 \times \P^1$ from Example \ref{ex:calabi1}. In turn, we have two inequivalent birational maps 
from the total space of the square root of $K_{\P^1 \times \P^1}$ on{to} the small resolution, related by a \textquotedblleft flop\textquotedblright.

Theorem \ref{crepant} abstractly proves the existence of a $1$-parameter family
of AC Calabi-Yau metrics of rate $-2$ on this manifold that are all asymptotic to one fixed Calabi-Yau cone metric $\omega_0$ on $C_1$ via $\pi$. This is because $H^2_c(M) = H_{2n-2}(E) = 0$
and $h^{1,1}_{\rm pr}(\P^1 \times \P^1) = 1$,\footnote{This argument also shows that complex cones over rank-$1$ Fano manifolds have no K{\"a}hler small resolutions.} and because we can combine the scaling action of $\R^+$ on our metrics and the diffeomorphism parameter $c$ into one single parameter that scales the exceptional $\P^1$, while leaving the metric unchanged at infinity. 

{By Theorem \ref{t:uniqueness}, this family must coincide with the explicit one from \cite{delaossa}}.
\end{example} 

\begin{example}[Goto]\label{ex:calabi2}
Calabi's metric $\omega$ on the total space of $K_D$ has rate $-2n$ and is $i\partial\bar{\partial}$-exact at infinity. 
Now $H^2(M) = H^2_c(M) \oplus H^{1,1}_{\rm pr}(D)$ by (\ref{e:resol_sequence}), and $H^2_c(M) = H_{2n-2}(E) = \R[\omega]$. 
Including one ``scale'' as above, $\omega$ thus moves in a $b_2(D)$-parameter family of AC Calabi-Yau metrics of rate $-2$.
\end{example}

\begin{remark}

(i) By \cite[Theorem 7.92]{Cheeger}, the $p^*\xi$ with $\xi \in \mathcal{H}^{1,1}_{{\rm pr},{\rm b}}(L)$ exhaust the set of leading terms of infinitesimal Ricci-flat deformations of $(C,g_0)$ with the complex structure held fixed.

(ii) If there exist any compactly supported K{\"a}hler classes at all, then $E$ must contain a divisor.
From \cite[\S 3.4]{vanC3}, the constant in front of $i\partial\bar{\partial}r^{2-2n}$ in (\ref{e:goto2}) is positive and proportional to $\langle \mathfrak{k}^n, [M]\rangle$. 
\end{remark}

\subsection{Small resolutions associated with flag manifolds}\label{s:flag-ex's}

The theme of this section is to obtain examples of small (and therefore crepant) resolutions of Calabi-Yau cones as total spaces of vector bundles. The following lemma gives us a general tool for making constructions of this kind.

\begin{lemma}\label{l:smallres}
Let $B$ be a compact complex manifold. Let $E \to B$ be a vector bundle of rank $r \geq 2$ and
let $p: \P(E) \to B$ denote its projectivization. Then, for any $k \in \N$, the following are equivalent:
\begin{enumerate}
\item $c_1(\P(E))$ is divisible by $k$.
\item Both $r$ and $c_1(E) + c_1(B)$ are divisible by $k$. 
\end{enumerate}
In this case, assume that $\P(E)$ is Fano and let $C$ denote the $k$-th root of $K_{\P(E)}$ with its zero section blown down. Then $C$ admits a small resolution by the total space of $E \otimes L$ for some $L \in {\rm Pic}(B)$ if and only if $k = r$. In this case, $L$ is unique, and is in fact given by $L^{-r} = \det E \otimes \det T_B$. 
\end{lemma}

\begin{proof}
The equivalence of (i) and (ii) is immediate from the identity 
\begin{equation}\label{e:c1projbdle}c_1(\P(E)) = r \xi_E + p^*(c_1(E) + c_1(B)),\end{equation}
where $\xi_E = c_1(\mathcal{O}_E(1))$. One can prove this by taking Chern characters in the exact sequences
\begin{align*}
0 \to T_{\P(E)/B} \to T_{\P(E)} \to p^*T_B \to 0, \\
0 \to \mathcal{O}_{\P(E)} \to \mathcal{O}_E(1) \otimes p^*E \to T_{\P(E)/B} \to 0,
\end{align*}
the second of which is simply a family version of the usual Euler sequence for $\P^{r-1}$.

By construction, for any vector bundle $F \to B$, the total space of $F$ is gotten by blowing down the fibers of $\P(F) \to B$ in the zero section of $\mathcal{O}_{F}(-1) \to \P(F)$. Recalling that 
$\P(E \otimes L) = \P(E)$ in a canonical fashion for all line bundles $L \to B$, we see that the total space of $E \otimes L$ resolves $C$ if and only if $K_{\P(E)} = \mathcal{O}_{E \otimes L}(-k)$. Since $\P(E)$ is Fano, this is equivalent to $c_1(\P(E)) = k \xi_{E \otimes L}$. 
Recall that $c_1(E \otimes L) = c_1(E) + r c_1(L)$. Together with (\ref{e:c1projbdle}), this {allows} us {to} rewrite our condition as $$(r -k)\xi_{E \otimes L} + p^*(c_1(E) + c_1(B) + r c_1(L)) = 0.$$
This is solvable if and only if $k = r$. Moreover, in this case, $L$ as in the claim is the unique solution because $p^*$ injects ${\rm Pic}(B)$ into ${\rm Pic}(\P(E))$, so that ${\rm Pic}(B)$ must be discrete as well. 
\end{proof} 

If $\P(E)$ is in addition K{\"a}hler-Einstein, then Theorem \ref{crepant} immediately gives us a $\rho_B$-dimensional family of AC Calabi-Yau metrics of rate $-2$ on the small resolution $E \otimes L$. Here, $\rho_B$ denotes the Picard rank of $B$, and we have used the fact that ${\rm Pic}(\P(E)) = \Z[\mathcal{O}_E(1)] \oplus p^*{\rm Pic}(B)$. 

On the other hand, taking branched covers {of the metrics} in Example \ref{ex:calabi2} with $r$-fold branching along the zero section of $K_{\P(E)}$, we also get a $(\rho_B + 1)$-parameter family of singular AC Calabi-Yau metrics on the total space of $\frac{1}{r}K_{\P(E)}$ with a cone angle of $2\pi r$ along the zero section.

\begin{conj}\label{conj:contr}
The singular {\rm AC} metrics on the $r$-th root of $K_{\P(E)}$ contract along one-parameter families to the smooth {\rm AC} metrics on $E \otimes L$ by shrinking the fibers of $p$ in the zero section.
\end{conj} 

We now discuss a new class of examples where a compact Lie group $G$ acts with cohomogeneity
one on the cone by holomorphic isometries, so that Theorem \ref{t:uniqueness} immediately tells us that the AC Calabi-Yau metrics on $E \otimes L$ have cohomogeneity one under $G$ as well. It would then be interesting to see whether the ODE that of necessity governs these metrics is explicitly solvable.

\begin{obs}
We can construct ${\rm AC}$ Calabi-Yau metrics of rate $-2$ and cohomogeneity one on a small resolution proceeding from any compact, simply connected, semisimple Lie group $G$, together with a pair of parabolic subgroups $P_1 \subsetneq P_2$ such that $P_2/P_1 = \P^{k-1}$, {where} $k$ divides $c_1(G/P_1)$.
\end{obs}

The {main point here is} that all flag manifolds $G/P$ are Fano with a $G$-invariant K{\"a}hler-Einstein metric \cite[\S 8]{besse}, and
$H^{{\rm odd}}(G/P,\Z) = 0$ \cite{borel}, so that the bundle $G/P_1 \to G/P_2$ can be written as $\P(E)$ for some vector bundle $E \to G/P_2$ by \cite[p.~515]{GriffithsHarris}. The condition on $c_1$ is combinatorially checkable, although in concrete examples $E$ may already be given to us as {a ``tautological''} vector bundle, in which case condition (ii) {of} Lemma \ref{l:smallres} {may} be more practical to check.

\begin{example}
The small resolution of the conifold from Example \ref{ex:cdo} is recovered by taking $G$ to be ${\rm SU}(2) \times {\rm SU}(2)$. {This group, being of rank $2$}, has two flag manifolds: $\P^1 \times \P^1$ and $\P^1$.
\end{example}

\begin{example}\label{ex:su3}
${\rm SU}(3)$ also has two flag manifolds:~the maximal one, SU$(3)/T^2 = \P(T^*\P^2)$ (of Fano index $2$, as for all groups $G$), and $\P^2$. 
The resulting $1$-parameter family of AC Calabi-Yau metrics on $T^*\P^2$ is originally due to Calabi \cite{Cal1}; these metrics are in fact hyper-K{\"a}hler. Conjecture \ref{conj:contr} was 
motivated 
by the fact that such contractions were recently constructed in this case \cite{bm, r}.
\end{example}

The example of SU$(3)$ can be generalized to higher dimensions in two different ways. We refer to \cite[8.111]{besse} for computations {involving} SU$(n+1)$ flags. The basic tool here is to label the space of all flags of the form $0 \subset V^{n_1} \subset V^{n_1 + n_2} \subset ... \subset \C^{n+1}$ by the ordered partition $(n_1,n_2,...)$ of $n + 1$.

\begin{example} 
The maximal flag variety $(1,...,1)$ of SU$(n+1)$ fibers over every flag variety of the form $(1,...,1,2,1,...,1)$ with $\P^1$ fibers. {Our construction applies here} because $c_1(G/T)$ is the sum of the positive roots for every group $G$, which is famously divisible by $2$ in the weight lattice.
\end{example}

\begin{example} The Grassmannian $G(k,n+1)$ of $k$-planes in $\C^{n+1}$ is given by $(k,n+1-k)$ and has at most four projective bundles sitting over it that are {also} flag manifolds: $(1,k-1,n+1-k)$, $(k-1,1,n+1-k)$, $(k, 1, n-k)$, $(k,n-k,1)$, with fibers $\P^{k-1}$, $\P^{k-1}$, $\P^{n-k}$, $\P^{n-k}$ {respectively}.

It is an instructive exercise to compute the divisibility of the first Chern classes of the projective bundles 
using roots. The positive roots are $\lambda_i - \lambda_j$ for $i < j$ with $\sum_{i=1}^{n+1} \lambda_i = 0$. To get $c_1(G/P)$, we need to sum those roots that do not vanish on the face of the Weyl chamber associated {to} $G/P$. {If} we think of the ordered partition in terms of \textquotedblleft grouping boxes\textquotedblright,
$$n+1 = (\underbrace{\Box \Box \Box \Box\Box\Box}_{n_1})(\underbrace{\Box\Box\Box\Box}_{n_2})(\underbrace{\Box\Box\Box\Box\Box}_{n_3})(...),$$
{this just} means that we need to sum all $\lambda_i - \lambda_j$ such that the $i$-th and $j$-th boxes belong to different groups. For example, we can now compute that $c_1(1,k-1,n+1-k) = (n + k)\lambda_1 + n(\lambda_2 + ... + \lambda_k)$, which is divisible by $k$ in the weight lattice if and only if $k|n$.

However, we can give a simpler and more complete discussion via Lemma \ref{l:smallres}(ii). The projective
bundles in question are clearly of the form $\P(E)$ for $E = T$, $T^*$, $Q$, $Q^*$, where $T$ denotes the rank $k$ tautological bundle and $Q = \underline{\C}^{n+1}/T$. Now, ${\rm Pic}(G(k,n+1)) = \Z[\det T]$ and $K_{G(k,n+1)} = (\det T)^{n+1}$, so Lemma \ref{l:smallres} immediately tells us that $T$ works iff $k|n$ (as we already know), 
$T^*$ works iff $k|n+2$, $Q$ works iff $n+1-k|n+2$, and $Q^*$ works iff $n+1-k|n$. Moreover, we learn that the line bundle $L$ that we need to twist $E$ by is $\det T$ to the power $\frac{n}{k}, \frac{n+2}{k}, \frac{n+2}{n+1-k}, \frac{n}{n+1-k}$, {respectively}. Finally, 
it is then also clear that we only get a $1$-parameter (scaling) family of Ricci-flat metrics.

As a special case, we can consider $G(1,n+1) = \P^n$. This gives us Ricci-flat metrics on the total space of $Q^* \otimes \mathcal{O}_{\P^n}(-1)$ as well as on the total space of $Q \otimes \mathcal{O}_{\P^2}(-2)$. Now, the fiber of $Q^* \otimes \mathcal{O}_{\P^n}(-1)$ over $\ell \in \P^n$ is naturally identified with ${\rm Hom}(\C^{n+1}/\ell, \C) \otimes \ell = {\rm Hom}(
\ell, \C^{n+1}/\ell)^*$, the cotangent space to $\P^n$ at $\ell$. If $n = 2$, this is in turn the same as the fiber of $Q \otimes \mathcal{O}_{\P^2}(-2)$, but not quite canonically {so} because one needs to choose an
element of $\wedge^{3,0}\C^3$. In any {case},
{the result is} that we recover Calabi's 
hyper-K{\"a}hler metrics on $T^*\P^n$ for every $n \geq 2$. Conjecture \ref{conj:contr} is in fact known {to hold true} in this case as well \cite{mal},
generalizing the work for $n = 2$ mentioned in Example \ref{ex:su3} {above}.
\end{example}

\begin{remark}
Flag manifolds and their characteristic classes are a classical topic in topology. For example, Borel-Hirzebruch \cite[p.~340]{BH} noticed that the Chern numbers $c_1^5$ of the SU$(4)$ flag manifolds $\P(T^*\P^3)$ and $\P(T\P^3)$ are different. {Thus, the complex structures of $\P(T^*\P^3)$ and $\P(T\P^3)$ cannot be homotopic, even though their underlying smooth manifolds are trivially diffeomorphic.}
\end{remark}
\newpage
\section{{Affine smoothings}}\label{smoothing}

\subsection{Overview}\label{s:smoothingov} We now study certain examples of AC Calabi-Yau manifolds which {are}, in some sense, the extreme opposites of crepant resolutions. The idea {here} is to realize a Calabi-Yau cone as an affine variety and to {``smooth out'' its} singularity by adding on
{terms of lower degree} to some of its defining equations.
The cone $C$ and its {associated ``smoothing''} $M$ are then diffeomorphic away from compact subsets, but {they are} no longer
canonically diffeomorphic or even biholomorphic.

We will only be considering \emph{regular} Calabi-Yau cones. Thus, the underlying variety $C = (\frac{1}{k}K_D)^\times$ for some $(n-1)$-dimensional Fano manifold $D$ and {some} $k|c_1(D)$, and the Ricci-flat cone metric on $C$ is lifted from a K{\"a}hler-Einstein metric on $D$ via the Calabi ansatz; compare Section \ref{s:calans}. By the arguments {in \cite{vanC4} used to prove Theorem \ref{t:affine}, we can assume that $C$ is realized as an algebraic variety in $\C^N$ for some $N$, in such a way that the canonical $\C^*$-action on $C$ (whose space of orbits is $D$) is the restriction of a diagonal $\C^*$-action on $\C^N$ with positive weights $w_1, ..., w_N > 0$.

It will be enough here to work with an intuitive notion of smoothing. The {most} basic example we have in mind is the $1$-parameter deformation $z_1^2 + ... + z_{N}^2 = t$ of the ordinary double point in $\C^{N}$, {with} $N = n+1$. Here, as in all {of} the other examples we {shall} be considering, $w_1 = ... = w_{N} = 1$ because $D$ is already projectively embedded by $\frac{1}{k}K_D$. In general, {in a smoothing,} the terms {that are} added on to the equations of $C$ {need} to be of sufficiently low degree relative to $w_1, ..., w_N$.

If a cone is rigid as a variety, then no such smoothings exist. For example, $K_D^\times$ is rigid if  $D = \P^{2}$ or ${\rm Bl}_1\P^2$ or any toric Fano manifold of dimension $n - 1 > 2$; compare \cite[(6.3), (9.1)]{Altmann}.
\begin{prop}\label{p:smoothing}
{\rm (i)} An $n$-dimensional smooth affine variety $M$ with trivial canonical bundle is a smoothing of
$C = (\frac{1}{k}K_D)^\times$ as above if and only if $M = X \setminus D$ for {some} $n$-dimensional Fano manifold $X$ of index at least $2$ containing $D$ as an anticanonical divisor such that $-K_X = (k + 1)[D]$. 

{\rm (ii)} $H^2_c(M) = H^2(M)$ if $n = 2$, and $H^2_c(M) = 0$ if $n > 2$. In both cases, $b^2(M) = b^2(X) - 1$, and the image of the K{\"a}hler cone of $X$ under restriction to $M$ is the whole space $H^2(M)$.
\end{prop}

If $n = 2$, then $X = \P^2$ {or} $\P^1 \times \P^1$. All Fano $3$-folds $X$ of index equal to $2$ are listed in \cite[\S 12.1]{AG5}. Up to deformation, every del Pezzo surface $D$ of degree at most $7$ occurs as an anticanonical divisor in {at least} one of these. Since {the cases} $D = \P^{2}$, ${\rm Bl}_1\P^2$ do not occur, we recover the fact mentioned earlier that $K_D^\times$ is not smoothable {for such $D$}. On the other hand, $D = {\rm Bl}_3\P^2$ occurs twice (in a hyperplane section of Segre($\P^2 \times \P^2$), and in $\P^1 \times \P^1 \times \P^1$), and the versal deformation space
of the singularity $K_D^\times$ is known to consist of two irreducible components \cite[(8.2), (9.1)]{Altmann}.

\begin{proof}[Proof of Proposition \ref{p:smoothing}]
(i) Given $M$, we can construct $X$ by passing to the completion of $M$ in the total space of $\mathcal{O}(1)$ over the weighted projective space $\P^{N-1}_{w_1,...,w_N}$, which naturally contains $\C^N\setminus 0$ as a dense open subset. Then $N_{D/X}$ is positive and $M = X \setminus D$ contains no compact analytic sets, so that {$D$ must be ample} {by} the Nakai-Moishezon criterion 
(see \cite[Proposition 10]{Goodman} for some more details on this type of argument). The relation $-K_X = (k + 1)[D]$ is clear by adjunction.

For the converse, we need to show that $X \setminus D$ really is a smoothing of $C = (N_{D/X}^*)^\times $. This follows from a more general construction. Given a compact complex manifold $X$ and a smooth divisor $D$, let $p:L \to X$ denote the {line bundle associated to $D$} and $s$ a defining section for $D$. Construct a hypersurface $X_t = \{v \in L: tv = s(p(v))\} $ for each $t \in \C$. This is isomorphic to $X$ for $t \neq 0$ and to the total space of $L|_D$ for $t = 0$. Notice that any two hypersurfaces in this family intersect precisely along $D$ in the zero section. If $L$ is positive, we can first compactify the total space of $L$ by adding a single point at infinity, 
and then remove the zero section, to construct an affine variety $V = (L^*)^\times$. 
Our family $X_t$ parametrized by $t \in \C$ then induces a subfamily of the product family $\C \times V$ such that all fibers over $\C^*$ are copies of $M$ whereas the central fiber is isomorphic to $(L^*|_D)^\times = C$. 

(ii) We have $H^2_c(M) = H^{2n-2}(M)$ by Poincar{\'e} duality, and for $n > 2$, {this vanishes} because $M$ is Stein; see \cite[Theorem 1]{AF}. It follows from a version of the Gysin sequence (the long exact sequence of the pair $(X,X\setminus D)$, combined with the Thom isomorphism for $N_{D/X}$) that
$$H^0(D) \to H^2(X) \to H^2(X \setminus D) \to H^1(D).$$
Thus, $H^2(X) \to H^2(M)$ is onto with a $1$-dimensional kernel. In particular, the image of the K{\"a}hler cone of $X$ in $H^2(M)$ is an open convex cone; but this cone also contains the origin because $[D]$ is a positive line bundle and the curvature forms of Hermitian metrics on $[D]$ are exact off $D$.
\end{proof}

We now focus on smoothings of \emph{complete intersection} Calabi-Yau cones, where, by the Lefschetz hyperplane theorem \cite[\S 13.2.3]{voisin2}, $b^2(M) = 0$ if $n > 2$. Our examples will in fact be diffeomorphic to the Milnor fiber of $C$, hence homotopy equivalent to a bouquet of $n$-spheres \cite[Satz 1.7(iv)]{Hamm}.

In Sections \ref{s:cubic}--\ref{s:ODP}, we explain, {through} a few examples, how one can construct a diffeomorphism $\Phi$ from $C$ to $M$ away from compact sets and estimate the rate of convergence $\lambda$ of the holomorphic volume forms on $M$ and $C$ with respect to $\Phi$. 
Theorem \ref{thm:main} then yields a $1$-parameter scaling family of $i\partial\bar{\partial}$-exact AC Calabi-Yau metrics of rate $\max \{-2n,\lambda\}$ on $M$, which, according to Theorem \ref{t:uniqueness}, contains all Calabi-Yau metrics on $M$ that are AC with respect to $\Phi$.

In the remaining sections, we then discuss an obvious gauge fixing issue: The rates of convergence to the cone that we compute depend upon the \emph{choice} of a diffeomorphism $\Phi$ between $M$ and $C$, and may therefore not be optimal.
Section \ref{gross3fold} collects some general thoughts and examples {regarding this matter}. In Section \ref{s:optimalrate}, we then prove that the rate $-2\frac{n}{n-1}$ that we obtain for the usual smoothing of the ordinary double point in $\C^{n+1}$ in Section \ref{s:ODP} is in fact optimal in a precise sense.}

\subsection{Example 1: Smoothings of a cubic cone in $\C^{4}$}\label{s:cubic}

We first consider the cubic cone 
$$C =\left\{z\in\C^{4}:\sum_{i=1}^{4}z_{i}^{3}=0\right\}.$$
Since this variety is simply the blowdown of the zero section of $K_{{\rm Bl}_6\P^2}$, using the Calabi ansatz, we can endow it with a Calabi-Yau cone structure $(g_{0}, \Omega_{0})$. By construction, the K\"ahler potential of $g_{0}$ is $\norm{\cdot}^{2/3}$, where $\norm{\cdot}$ is the norm induced on $K_{{\rm Bl}_6\P^2}$ by the K\"ahler-Einstein metric on ${\rm Bl}_6\P^2$. As for $\Omega_{0}$, it is the restriction to $C$ of the $(3,0)$-form $\Omega_0$ on $\C^4$ defined implicitly by the equation
\begin{equation}\label{hotholo}
dz_{1}\wedge ... \wedge dz_{4}=\Omega_{0}\wedge d\left(\sum_{i=1}^{4}z_{i}^{3}\right).
\end{equation}
We will consider one particular smoothing of $C$, {namely} the smooth affine variety $M$ defined by $$M =\left\{z\in\C^{4}:\sum_{i=1}^{4}z_{i}^{3}=1\right\}.$$
{Then $M$ carries a holomorphic volume form $\Omega$, defined by \eqref{hotholo} with $\Omega$ in place of $\Omega_{0}$. Estimating the rate of convergence of $\Omega$ and $\Omega_0$ with respect to $g_0$ and a diffeomorphism $\Phi$ we are yet to construct, and invoking Theorems \ref{thm:main} and \ref{t:uniqueness}, {allows us to} prove the following proposition.

\begin{prop}\label{cor:1}
For each $c>0$, $M$ admits an $i\p\bar{\p}$-exact {\rm AC} Calabi-Yau metric $\omega_c$ such that
\begin{equation}\label{e:cor1}
\Phi^*\omega_c - c\omega_0 =  const \cdot i\partial\bar{\partial}r^{-4} + O(r^{-7-\epsilon})
\end{equation}
for some $\epsilon>0$, {with} $\omega_{c_2} = \frac{c_2}{c_1}\omega_{c_1}$ for all $c_1, c_2 > 0$. Moreover, these are the only Calabi-Yau metrics on $M$ that are {\rm AC} with respect to the chosen diffeomorphism $\Phi$.
\end{prop}

\begin{remark}\label{r:cor1}
Even modulo biholomorphism, $C$ admits many more affine smoothings than just $M$:
Every smoothing of $C$ in our sense is in particular a deformation of $C$. According to 
\cite{Kas},
\begin{equation}\label{e:cubic_versal}
t z_{1}z_{2}z_{3}z_{4}+\sum_{i=1}^{4}z_{i}^{3}+\sum_{1\leq i<j<k\leq 4}t_{ijk}z_{i}z_{j}z_{k}
+\sum_{1\leq i<j\leq 4}t_{ij}z_{i}z_{j}+\sum_{i=1}^{4}t_{i}z_{i}=\epsilon
\end{equation}
cuts out a versal deformation of $C$ in $\C^{16}\times\C^{4}$. We are only interested in the smooth fibers of the subfamily {given} by $t = 0$, $t_{ijk} = 0$, which are homotopy equivalent to $\vee ^{16} S^3$ \cite[Korollar 3.10]{GH} and pairwise diffeomorphic. Proposition \ref{cor:1} holds for each of these except for the {statement of the} rates in (\ref{e:cor1}). {As it turns out}, $\Omega$ converges to $\Omega_0$ at rate $-6$ if all $t_{ij} = 0$, and at rate $-3$ in general.
\end{remark}

\begin{remark}\label{r:higher_dim_cubic}
By \cite[Proposition 3.1]{Arezzo}, the cubic hypersurface $\sum_{i=1}^{n+1}z_{i}^{3}=0$ in $\P^{n}$ admits a K\"ahler-Einstein metric for any $n\geq 3$, so that, via the Calabi ansatz, the cubic cone $\sum_{i=1}^{n+1}z_{i}^{3}=0$ in $\C^{n+1}$ is seen to be Calabi-Yau. For any smoothing of this cone involving constant, linear, or quadratic terms as in (\ref{e:cubic_versal}), one can compute that the associated holomorphic volume forms converge at rate $-3\frac{n}{n-2}$, 
$-2\frac{n}{n-2}$, or $-\frac{n}{n-2}$ respectively. Thus, if $n > 3$, then Proposition \ref{cor:1} holds verbatim with these rates. For $n > 4$ we get rates slower than $-2$, and so the full strength of Theorem \ref{thm:main} is needed.
\end{remark}

\subsubsection{Construction of a diffeomorphism} The idea here is to project $C$ orthogonally onto $M$ in $\C^4$ to {construct} a diffeomorphism $\Phi$ from the complement of a compact subset of $C$ onto the complement of a compact subset of $M$. Using Lemma \ref{l:normproj} below, this mapping takes the form
$$\Phi(z_{1},...,z_{4}) =(z_{1} + \alpha(z)\bar{z}^{2}_{1},...,z_{4}+\alpha(z)\bar{z}^{2}_{4})$$
away from some {sufficiently large closed ball $\bar{B}_R\subset\C^4$ of radius $R$ centered at the origin}. 

\begin{lemma}\label{l:normproj}
There exists $R>0$ and a smooth function
$\alpha:\C^{4}\setminus
\bar{B}_{R}\to\C$ with
$\alpha(z) \sim |z|^{-4}$ such that 
$\sum_{i=1}^{4}(z_{i}+\alpha(z)\bar{z}_{i}^{2})^{3}=1$
for every $z\in C\setminus
\bar{B}_{R}$.
\end{lemma}

\begin{proof}
Taking complex coordinates $z=(z_{1},...,z_{4})$ on $S^7 \subset \C^{4}$, define a function $f$ by
\begin{equation*}
f:S^{7}\times[0,\infty)\times\C\to\C,
\;\, f(z,r,y)=
3y\left(\sum_{i=1}^{4}|z_{i}|^{4}\right)+3y^{2}\left(\sum_{i=1}^{4}|z_{i}|^{2}\bar{z}^{3}_{i}\right)+
y^{3}
\left(\sum_{i=1}^{4}\bar{z}_{i}^{6}\right)-r^{3},
\end{equation*}
and fix any point $p=(p_{1},...,p_{4})\in S^{7}$. Since $$f(p,0,0)=0,\;\,
\frac{\p f}{\p y}(p,0,0)=3\sum_{i=1}^{4}|p_{i}|^{4}\neq 0,$$ the implicit function theorem asserts the existence of a unique
smooth function $s_{p}$, defined in some open neighborhood
$U_{p} \times [0,\epsilon_p)$ of $(p,0),$ such that
$s_{p}(p,0)=0$ and $f(z,r,s_{p}(z,r))=0$.
An obvious covering argument on $S^{7}$ then yields a smooth function $s: S^7 \times [0,\epsilon) \to \C$ satisfying
$$s(z,0)=0, \quad f(z,r,s(z,r))=0.$$ 
We now set $R=\epsilon^{-1}$ and define $\alpha:\C^{4}\setminus \bar{B}_{R}\to\C$ by
$$\alpha(z) = \frac{1}{|z|}s\left(\frac{z}{|z|},\frac{1}{|z|}\right).$$
The fact that $\alpha$ satisfies $\sum_{i=1}^{4}(z_{i}+\alpha(z)\bar{z}_{i}^{2})^{3}=1$
for each $z\in C\setminus
\bar{B}_{R}$ is straightforward to verify. In order to complete the proof of the lemma, we need only show that $\alpha(z) \sim |z|^{-4}$.

To see this, observe that $\alpha(z)|z|^{4}P(z)=1$ on $\C^{4}\setminus\bar{B}_R$, where
\begin{equation*}
P(z)=3\left(\sum_{i=1}^{4}\frac{|z_{i}|^{4}}{|z|^{4}}\right)+3(\alpha(z)|z|)\left(\sum_{i=1}^{4}\frac{|z_{i}|^{2}\bar{z}^{3}_{i}}{|z|^{5}}\right)+(\alpha(z)|z|)^{2}
\left(\sum_{i=1}^{4}\frac{\bar{z}_{i}^{6}}{|z|^{6}}\right).
\end{equation*}
It then suffices to note that $\alpha(z)|z| = s(\frac{z}{|z|},\frac{1}{|z|}) \to 0$ uniformly as   $|z| \to \infty$.
\end{proof}

\subsubsection{Some preliminary estimates}
We next take conical coordinates $(r,x)$ on $C = \R^+ \times L$, where $r(z)=\norm{z}^{1/3}$ is the $g_0$-distance from $z$ to the apex of $C$, and define diffeomorphisms
$$
\nu_{t}: L\times[1,2]\ni(x,r)\mapsto
(x,tr)\in L\times[t,2t].
$$
We wish to write $\nu_t$ in terms of Cartesian coordinates on $\C^{4}$.

The basic observation is that $\nu_{t}(z)=t^\mu z$ for some $\mu > 0$ because $C$ is homogeneously embedded, or in other words because $-K_{{\rm Bl}_6\P^2}$ is very ample.
But then we must have $\mu = 3$ because
\begin{equation*}
t\norm{z}^{\frac{1}{3}}=tr(z)=r(\nu_{t}(z))=\norm{\nu_{t}(z)}^{\frac{1}{3}}=\norm{t^\mu z}^{\frac{1}{3}}
=t^{\frac{\mu}{3}}\norm{z}^{\frac{1}{3}}.
\end{equation*}
It now follows immediately that $|z_i| \sim r^3$ for all $i$, though the precise proportionality depends upon the K{\"a}hler-Einstein metric on ${\rm Bl}_6\P^2$. {In addition,} $z_i = O(r^3)$ with $g_0$-derivatives by Lemma \ref{simple321}.

Observe that $\alpha \sim r^{-12}$. We also require estimates on $d\alpha$ and {on} its covariant derivatives. Since $\alpha(z)|z|^4P(z) = 1$, we find that $|d\alpha|_{g_0} \leq C r^{-12}(r^{-1} + |dP|_{g_0})$. {Now}, {using the definition of $P$,} 
$$dP = Q_{-1} + (P_{-12}Q_{2} + P_3 d\alpha) + (P_{-24}Q_5 + P_{-6}d\alpha),$$
where $P_k$ (respectively $Q_k$) denotes a function (respectively $1$-form) that we already know is $O(r^k)$. {Hence}, $|d\alpha|_{g_{0}}=O(r^{-13})$.
More generally, {we have that} $\alpha=O(r^{-12})$ with $g_{0}$-derivatives.

\subsubsection{Computation of the asymptotics of the holomorphic volume forms} 
Using these preliminary estimates, we are now able to compute the rate of convergence of the holomorphic volume forms $\Omega$ on $M$ and $\Omega_0$ on $C$ with respect to our diffeomorphism $\Phi$. 
First observe that we may write
$$
\Omega_{0}=\left.(-1)^{i-4}\frac{dz_{1}\wedge ...\wedge\widehat{dz_{i}} \wedge ...\wedge dz_{4}}{3z_{i}^{2}}\right|_{C},
$$
and likewise for $\Omega$,
on the open subsets $C \cap \{z_{i}\neq0\}$ and $M \cap \{z_{i}\neq0\}$ respectively.

Next, consider the set $A_{4}=\{z\in C:|z_{4}|>\frac{1}{3}|z|>R\}$, where $R\gg 1$ is chosen sufficiently large so that $\frac{1}{3}|z| > |\alpha(z)||z|^{2}$ for all $z\in C$ with $|z|>R$. {Observe that} $z \in A_4$ implies that $z_{4}+\alpha(z)\bar{z}_{4}^{2}\neq 0$. 
As a consequence, we find that in a neighborhood of any point $z\in A_{4}$,
\begin{equation*}
\Phi^{*}\Omega=\left.\frac{d(z_{1}+\alpha(z)\bar{z}_{1}^{2}) \wedge
d(z_{2}+\alpha(z)\bar{z}_{2}^{2})\wedge
d(z_{3}+\alpha(z)\bar{z}_{3}^{2})}{3(z_{4}+\alpha(z)\bar{z}_{4}^{2})^{2}}\right|_{C},
\end{equation*}
so that, using our preliminary estimates,
\begin{equation*}
\Phi^{*}\Omega=\left.\frac{(dz_{1}+O(r^{-7}))\wedge
(dz_{2}+O(r^{-7}))\wedge(dz_{3}+O(r^{-7}))}{3z^{2}_{4}
(1+O(r^{-9})
)^{2}
}
\right|_{C}
=\Omega_{0}+O(r^{-9}).
\end{equation*}
This estimate works equally well on $A_{i}=\{z\in C:|z_{i}|>\frac{1}{3}|z|>R\}$ for all $i \neq 4$. In light of the fact that the $A_i$ cover $C \setminus \bar{B}_{3R}$, we deduce that $\Phi^{*}\Omega-\Omega_{0}=O(r^{-9})$ on $C\setminus\bar{B}_{3R}$ with $g_{0}$-derivatives.

\subsection{Example 2: Smoothings of the intersection of two quadric cones in $\C^{5}$}\label{s:quadrics} Our second
example is based on the complete intersection cone
\begin{equation*}
C =\left\{z\in\C^{5}:f_{1}(z)=\sum_{i=1}^{5}z_{i}^{2}=0,\;\, 
f_{2}(z)=\sum_{i=1}^{5}\lambda_{i}z^{2}_{i}=0\right\},
\end{equation*}
where $\lambda_{i}\neq\lambda_{j}$ for $i\neq j$. Let $Q_{1}, Q_{2}$ be the 
projectivizations in $\P^{4}$ of the quadrics defining $C$.  Then $Q_{1}\cap Q_{2} = {\rm Bl}_5\P^2$ is an anticanonically embedded del Pezzo surface, so that $C$ realizes the blowdown of the zero section of $K_{Q_{1}\cap Q_{2}}$. The Calabi ansatz provides us with a Calabi-Yau cone metric $g_{0}$ {on $C$} with K{\"a}hler potential $\norm{\cdot}^{2/3}$, where $\norm{\cdot}$ is the norm induced on $K_{Q_{1}\cap Q_{2}}$ by the K\"ahler-Einstein metric on $Q_{1}\cap Q_{2}$. The cone $C$ also admits a holomorphic volume form $\Omega_{0}$, defined by $$dz_{1}\wedge ... \wedge dz_{5}=\Omega_{0}\wedge df_{1}\wedge df_{2}$$ (that is, by the Poincar\'e residue formula), whose length with respect to $g_{0}$ is constant.

We take the smoothing 
$M =\{z\in\C^{5}:f_{1}(z)=f_{2}(z)=1\}$ of $C$, which also admits a holomorphic volume form $\Omega$ satisfying $dz_{1}\wedge ... \wedge dz_{5}=\Omega\wedge df_{1}\wedge df_{2}$.
Using arguments as in Section \ref{s:cubic}, we then construct an orthogonal projection diffeomorphism $\Phi$ of the form $\Phi(z)_i = z_i + (\alpha(z) + \bar{\lambda}_i \beta(z))\bar{z}_i$. It turns out that this yields an estimate of $-6$ for the rate of convergence of  $\Omega$ towards $\Omega_0$.

\begin{prop}\label{cor:2}
For each $c>0$, $M$ admits an $i\p\bar{\p}$-exact {\rm AC} Calabi-Yau metric $\omega_c$ such that
\begin{equation}\label{e:cor:2}
\Phi^*\omega_c - c\omega_0 = O_\delta(r^{-6+\delta})
\end{equation}
with $g_0$-derivatives for every $\delta > 0$, and such that $\omega_{c_2} = \frac{c_2}{c_1}\omega_{c_1}$ for all $c_1, c_2 > 0$. {These} are again the only Calabi-Yau metrics on $M$ that are {\rm AC} with respect to the chosen diffeomorphism $\Phi$.
\end{prop}

\begin{remark}\label{r:cor2}
In analogy with Remark \ref{r:cor1}, one can use \cite{Kas} to compute that the following equations cut out a
versal deformation of the cone $C$ in $\C^9 \times \C^5$:

\begin{equation}\label{e:versal_quadrics}
\alpha z_{4}^{2}
+\beta z_{5}^{2}+\sum_{i=1}^{5}z_{i}^{2}+\sum_{i=1}^{5}t_{i}z_{i}=\epsilon_{1}, \;\,\sum_{i=1}^{5}\lambda_{i}z^{2}_{i}=\epsilon_{2}.
\end{equation}
The smoothings in our sense are exactly the smooth fibers of the subfamily $\alpha = \beta = 0$. They are all diffeomorphic to each other, and homotopy
equivalent to $\vee^9S^3$. Proposition \ref{cor:2} holds true verbatim, except
that the rate $-6 + \delta$ in (\ref{e:cor:2}) {must} be replaced by $-3$ if at least one of the $t_i \neq 0$.
\end{remark}

\begin{remark}\label{r:higher_dim_quadrics}
By \cite[Corollary 3.1]{Arezzo}, the intersection of the quadrics $\sum_{i=1}^{n+2}z_{i}^{2}=0$,  $\sum_{i=1}^{n+2}\lambda_{i}z_{i}^{2}=0$ (with $\lambda_{i}\neq\lambda_{j}$ for $i\neq j$) in $\P^{n+1}$ admits a K\"ahler-Einstein metric for all $n \geq 3$, so the corresponding affine cone in $\C^{n+2}$ is Calabi-Yau. For smoothings of this cone involving constant or linear terms as in (\ref{e:versal_quadrics}), the associated holomorphic volume forms converge at rate $-2\frac{n}{n-2}$ or $-\frac{n}{n-2}$ respectively; if $n > 3$, then Proposition \ref{cor:2} holds verbatim with these rates. As in Remark \ref{r:higher_dim_cubic}, the rates are slower than $-2$ in most cases, but again always strictly faster than $-1$.
\end{remark}

\subsection{Example 3: The smoothing of the ordinary double point in $\C^{n+1}$}\label{s:ODP}
Our final example is easier, and the metrics we construct are in fact not new. The cone will simply be given by
$$C=\left\{z\in\C^{n+1}:f(z)=\sum_{i=1}^{n+1}z_{i}^{2}=0\right\},$$
the ordinary double point in $\C^{n+1}$. This can be viewed as a regular Calabi-Yau cone $(\frac{1}{k}K_D)^\times$ over the hyperquadric $D \subset \P^n$ with its unique SO$(n+1)$-invariant K{\"a}hler-Einstein metric. Notice that $D$ has index $n-1$ if $n > 2$ and is then embedded by the maximal root of $-K_D$, so that $k = n-1$. We will be working with the particular smoothing $M =\{z\in\C^{n+1}:f(z)=1\} \cong T^*S^n$.

This configuration is quite special compared to the examples in Sections \ref{s:cubic}--\ref{s:quadrics} for two reasons. First, the deformation space of $C$ is only $1$-dimensional now, with all smooth fibers biholomorphic to $M$. Bearing in mind {the fact} that $b^2(M) = 0$ if $n > 2$, this means that the AC Calabi-Yau metric we will construct is \emph{rigid} in the strongest possible sense. Second, the canonical SO$(n+1)$-action
on $\C^{n+1}$ preserves {both} $M$ and $C$, as well as the cone metric on $C$. {So} by Theorem \ref{t:uniqueness}, our {AC} metric {on $M$ has} to be SO$(n+1)$-\emph{invariant}. The relevant normal projection map $\Phi$ from $C$ to $M$ that we use to estimate {the} rates is given by the SO$(n+1)$-equivariant function $\Phi(z) = z + \frac{\bar{z}}{2|z|^2}$.

\begin{prop}\label{cor:3}
{If $n > 2$, then $M$ admits $i\p\bar{\p}$-exact {\rm AC} Calabi-Yau metrics $\omega_c$ such that
\begin{equation}\label{e:cor:3}
\Phi^*\omega_c - c\omega_0 = O(r^{-2\frac{n}{n-1}})
\end{equation}
and $\omega_{c_2} = \frac{c_2}{c_1}\omega_{c_1}$ for all $c_{1}, c_{2} > 0$. The metrics} $\omega_c$ are invariant under the obvious ${\rm SO}(n+1)$-action on $\C^{n+1}$, and are the only Calabi-Yau metrics on $M$ that are {\rm AC} with respect to $\Phi$.
\end{prop}

According to \cite[Lemma 5]{Stenzel}, this means that we have reproduced Stenzel's metric on $T^*S^n$. We will revisit this point in Section \ref{s:optimalrate} and use Stenzel's explicit formula to prove that the rate in (\ref{e:cor:3}) cannot be improved {upon} by deforming our chosen diffeomorphism $\Phi$.

\begin{remark}\label{r:ehrem} The case $n = 2$ is somewhat exceptional. We have $H^2_c(M) = H^2(M) = \R$, and every class is K{\"a}hler. Our theory then formally implies that each class in $H^2(M)$ contains a $1$-parameter family of AC Calabi-Yau metrics of rate $-4 + \delta$ for every $\delta > 0$. These are of course all isometric to some rescaling of Eguchi-Hanson. In particular, the rate is really $-4$. However, our family contains one more degree of freedom besides scaling: Identifying $M$ with $T^*\P^1$ by a hyper-K{\"a}hler rotation,
one sees that the harmonic vector fields $X_c$ defined after Corollary \ref{c:whatisc} are all equal to the scaling field $X = r\partial_r$ along the fibers of $T ^*\P^1$. The flow of $X$ accounts for the extra parameter.
\end{remark}

\subsection{Discussion of the gauge fixing issue}\label{gross3fold} 
The estimates on the rate of convergence of $g$ to $g_0$ in our examples {required the choice} of a diffeomorphism $\Phi$ between $C$ and $M$ near infinity. We
made a rather natural choice for $\Phi$, but conceivably the rates {we obtain} could be improved upon by applying a suitable 
gauge fixing. In this section, we collect {together} some simple remarks on this issue.

(1) Consider the ordinary double point $C$ from Section \ref{s:ODP}. If one smooths $C$ by adding on both constant and \emph{linear} terms to its defining equation and applies our theory, then one finds that the rate of convergence of the resulting AC Calabi-Yau metric to its asymptotic cone is at least $-\frac{n}{n-1}$. But every such smoothing is biholomorphic to the standard one considered
in Section \ref{s:ODP}, so these metrics are still isometric to Stenzel, {where} we know that the rate is in fact at least $-2\frac{n}{n-1}$.

(2) If $n = 3$, there exists a connection (though not an entirely rigorous one) between the values of the rates and certain results in algebraic geometry: Consider a $3$-dimensional \emph{compact} Calabi-Yau variety $Y$ with isolated singularities. It is reasonably well understood under what conditions $Y$ can or cannot be deformed to a smooth Calabi-Yau $3$-fold; see e.g. 
~\cite{Friedman, Gross}. On the other hand, if all the singularities of $Y$ are locally isomorphic to Calabi-Yau cones $C_i$, then $Y$ is widely expected to admit Calabi-Yau \emph{metrics} with conical singularities modelled on the $C_i$ \cite[Definition 4.6]{Chan}.
Assuming the existence of such metrics, Chan \cite{Chan} used gluing arguments to prove that $Y$ is indeed smoothable if there exist AC Calabi-Yau manifolds $M_i$ asymptotic to $C_i$ at rate \emph{strictly faster} than $-3$.

(2a) Our computations yield the critical rate $-3$ for AC smoothings of the ordinary double point, as was already pointed out by Chan. On the other hand, there do exist obstructions to smoothing
Calabi-Yau varieties $Y$ with ordinary double points \cite{Friedman}. This indicates that $-3$ is optimal.

(2b) In Sections \ref{s:cubic}--\ref{s:quadrics}, we found AC smoothings of rate $-3$ as well as $-6$ for certain complete intersection $3$-fold cones. By \cite[Theorem 3.8]{Gross}, compact Calabi-Yau varieties $Y$ with singularities locally isomorphic to any of these cones are in fact always smoothable, for algebraic reasons. 

(3) Everything we {have} said so far is consistent with conjecturing that, for a complete intersection {Calabi-Yau cone} $C$, the optimal rate for our AC Calabi-Yau metric on an affine smoothing $M$ of $C$ can be {computed} by realizing $M$ as a member of the particular versal deformation of $C$ constructed in \cite[p.~24]{Kas}; cf.~Remarks \ref{r:cor1}--\ref{r:higher_dim_cubic} and \ref{r:cor2}--\ref{r:higher_dim_quadrics}. But we have no evidence for this {conjecture} except {that it holds true for the ordinary double point, as we will see in Section \ref{s:optimalrate} below.}

\subsection{The optimal rate for smoothing the ordinary double point}\label{s:optimalrate}
We close by computing the leading term in the asymptotic expansion of Stenzel's Ricci-flat metric $g$ on $\sum_{i=1}^{n+1} z_i^2 = 1$ relative 
to its tangent cone metric $g_0$ on $\sum_{i=1}^{n+1} z_i^2 = 0$. We use the same diffeomorphism $\Phi(z) =$ $z + \frac{\bar{z}}{2|z|^2}$ as in
Section \ref{s:ODP} to identify Stenzel and its tangent cone away from compact subsets of each.

\begin{prop}\label{p:stenzelrate}
If $n > 2$, then $\Phi^*g - g_0$ is equal to the tracefree symmetric bilinear form \eqref{e:stenzel} to leading order. This has rate $-2\frac{n}{n-1}$ 
and satisfies the Bianchi gauge condition relative to $g_0$.
\end{prop}

\begin{remark}\label{r:ct1} This result contradicts \cite[Theorem 0.16]{Cheeger}. To see why, notice that the proof given in \cite{Cheeger} really only requires that $M$ is AC, i.e.~that the \emph{conclusions} of \cite[Theorem 0.15]{Cheeger} hold for $M$. Moreover, the argument is based on analyzing the linearized Ricci-flat equations on $C$ in a Bianchi gauge. {Finally,} as
explained at the end
of Section \ref{s:smoothingov}, the
optimal rates for $g$ and $J$ in this example are a priori equal. See Remark \ref{r:ct2} {below} for further details about the proof in \cite{Cheeger}.
\end{remark}

\begin{remark}
The $n = 2$ case is Eguchi-Hanson, whose leading term can be written as $i\partial\bar{\partial}r^{-2}$ on $\C^2/\Z_2$. This tensor has rate $-4$ and is obviously in Bianchi gauge relative to the flat metric.
\end{remark}

The proof of Proposition \ref{p:stenzelrate} is an explicit computation based on the following lemma. One could also use the presentation of the Stenzel metric in terms of left-invariant $1$-forms on ${\rm SO}(n+1)$ given in \cite{cvetic}; compare in particular the shapes of (\ref{e:stenzel}) and  \cite[(2.34)]{cvetic}.

\begin{lemma}[Stenzel \cite{Stenzel}]\label{l:stenzel}
Up to scaling, the unique {\rm SO}$(n+1)$-invariant {\rm AC} Calabi-Yau metric $g$
on the standard smoothing of the ordinary double point in $\C^{n+1}$ is given by $\omega = i\partial\bar{\partial}(f \circ \tau)$, where 
$\tau(z) = |z|^2$ in $\C^{n+1}$, and $f(\tau) = h(w)$ with $\tau = \cosh w$, $(h'(w)^n)' = (\sinh w)^{n-1}$, and $h'(0) = 0$. 
\end{lemma}

One checks that $f(\tau) = C_n \tau^{1-(1/n)}(1 + k(\tau))$ with $k(\tau) \sim \tau^{-1}$ if $n = 2$, 
$k(\tau) \sim \tau^{-2}\log \tau$ if $n = 3$, and $k(\tau) \sim \tau^{-2}$ if $n \geq 4$.
The Calabi-Yau cone metric $g_0$ on the ordinary double point which serves as an asymptotic model for the Stenzel metric $g$ has K{\"a}hler potential $\tau^{1-(1/n)}$ up to scaling.

\begin{proof}[Proof of Proposition \ref{p:stenzelrate}]
Since $n \geq 3$, Lemma \ref{l:stenzel} and the above tell us that, up to scale, $g$ is equal to $\hat{g} = (1 + O(\tau^{-2}\log\tau))\hat{g}_0$ restricted from $\C^{n+1}$ to the quadric, where $\hat{g}_0 = \tau^{-(1/n)}G$ and
\begin{equation}\label{e:metric}
G = \sum {\rm Re}(G_{ij})(dx_i \otimes dx_j + dy_i \otimes dy_j) + {\rm Im}(G_{ij})(dx_i\otimes dy_j - dy_i \otimes dx_j)
\end{equation}
is the real symmetric bilinear form on $\C^{n+1}$ associated with the positive Hermitian matrix $$G_{ij} = \delta_{ij} -\frac{\bar{z}_i z_j}{n\tau}.$$

We now need to pull $\hat{g}$ {back} by $\Phi = {\rm id} + \Psi$, with $\Psi(z) = \frac{\bar{z}}{2\tau}$,
then subtract $\hat{g}_0$, and finally restrict to the tangent bundle of the quadric cone. 
Since $\tau \circ \Phi = \tau(1 + (2\tau)^{-2})$ and since we eventually only care 
about $O(\tau^{-1})$ relative errors, it will be enough to deal with $G$, treating
$\tau$ as a constant.

Some preparations first: By SO$(n+1)$-invariance, it suffices to work at $p_0 = (\frac{1}{\sqrt{2}},\frac{i}{\sqrt{2}},0,...,0)\sqrt{\tau}$ in the quadric cone. At this point, the tangent spaces to the cone and its link are given by 
\begin{align*}
T_{\rm cone} &= \ker\left(\sum x_l dx_l - y_ldy_l\right) \cap \ker\left(\sum x_l dy_l + y_l dx_l\right) = {\rm span}(\partial_{x_1} + \partial_{y_2}, \partial_{y_1} - \partial_{x_2}) \oplus \R^{2(n-1)}_{\rm last},\\
T_{\rm link} &= \ker(d\tau) \cap T_{\rm cone} = \ker(dx_1 + dy_2) \cap T_{\rm cone} =  {\rm span}(\partial_{y_1} - \partial_{x_2}) \oplus \R^{2(n-1)}_{\rm last}.
\end{align*}

We can then determine $(({\rm id} + \Psi)^* - {\rm id}^*)G$ to leading order, $O(\tau^{-1})$,
by computing the action of
$({\rm id} + \Psi)^* - {\rm id}^*$ 
on the right-hand side of (\ref{e:metric}) using the Leibniz rule. The first contribution, gotten by
acting on the ${\rm Re}(G_{ij})$ and ${\rm Im}(G_{ij})$ coefficients, is given by
$$-\frac{1}{n\tau^2} \sum {\rm Re}(z_iz_j)(dx_i \otimes dx_j + dy_i \otimes dy_j).$$
At $p_0$, it is easy to see that the restriction of this to $T_{\rm cone}$ vanishes, so we can just ignore it. As for the terms gotten by letting $({\rm id} + \Psi)^* - {\rm id}^*$ act on the differentials, we first observe that, at $p_0$,
\begin{align*}
\Psi^*dx_i &= d\frac{x_i}{2\tau} = \frac{dx_i}{2\tau} - \frac{x_i d\tau}{2\tau^2} = \frac{dx_i}{2\tau} - \delta_{1i}\frac{dx_1 + dy_2}{2\tau},\\
\Psi^*dy_i &= -d\frac{y_i}{2\tau} = -\frac{dy_i}{2\tau} + \frac{y_i d\tau}{2\tau^2} = -\frac{dy_i}{2\tau} + \delta_{2i}\frac{dx_1 + dy_2}{2\tau}.
\end{align*}
With this in hand, a lengthy but completely straightforward computation shows that the remaining 
$O(\tau^{-1})$ contributions to $(({\rm id} + \Psi)^* - {\rm id}^*)G$ are given by
$$
\frac{1}{\tau}\left(\begin{pmatrix} -\frac{1}{2n} & 0 & 0 & 0  \\
0 & \frac{1}{2n}-1 & 0 & 0 \\
0 & 0 & 1-\frac{1}{2n} & 0 \\
0 & 0 & 0 & \frac{1}{2n}
\end{pmatrix} \oplus {\rm id}_{n-1} \otimes \begin{pmatrix}1 & 0 \\ 0 & -1\end{pmatrix}\right)
$$
in terms of the real basis $\partial_{x_1}, \partial_{y_1}, \partial_{x_2}, \partial_{y_2}, (\partial_{x_i}, \partial_{y_i})_{i = 3}^{n+1}$ of $\C^{n+1}$.

The final step in computing $h := \Phi^*g - g_0$ modulo $o(\tau^{-1})$ is to restrict the ensuing expression for the leading term of $\Phi^*\hat{g} - \hat{g}_0$ to $T_{\rm cone}$ and {then} diagonalize with respect to $g_0 = \tau^{-(1/n)}G|_{T_{\rm cone}}$.
This is not difficult if one bears in mind {the fact that}
$$G = \begin{pmatrix} 1-\frac{1}{2n} & 0 & 0 & -\frac{1}{2n}  \\
0 & 1-\frac{1}{2n} & \frac{1}{2n} & 0 \\
0 & \frac{1}{2n} & 1-\frac{1}{2n} & 0 \\
-\frac{1}{2n} & 0 & 0 & 1-\frac{1}{2n}\end{pmatrix} \oplus  {\rm id}_{2(n-1)}.$$
{The end result is that}, at our chosen point $p_0$, {$h$ takes the form}
\begin{equation}\label{e:stenzel}h=\frac{1}{\tau}{\rm diag}\left(\begin{pmatrix}0 & 0 \\ 0 & 0\end{pmatrix},\begin{pmatrix}1 & 0 \\ 0 & -1\end{pmatrix}, ..., \begin{pmatrix}1 & 0 \\ 0 & -1\end{pmatrix}\right), \;\,\tau^{1-\frac{1}{n}} = r^2,\end{equation}
where $r$ denotes the radius function of the Ricci-flat cone metric $g_0$ and the matrix representation is with respect to a $g_0$-orthonormal basis $(\partial_r, J\partial_r), (X_3, Y_3), ..., (X_{n+1}, Y_{n+1})$ of $T_{\rm cone}$ such that the 
last $2n-2$ vectors form a common rescaling of the Euclidean basis $\partial_{x_3}, \partial_{y_3}, ..., \partial_{x_{n+1}}, \partial_{y_{n+1}}$.

To finish the proof, we must compute ${\rm div}_{g_0}(h)$. {Since} this is an SO$(n+1)$-invariant $1$-form on the cone, {it} can be written as $\alpha(r) dr + \beta(r) d^c r$. Indeed, the stabilizer of $p_0$ in SO$(n+1)$ is the standard SO$(n-1)$, so that $J\partial_r$ is the only isotropy invariant tangent vector to the link at $p_0$, hence the only vector at $p_{0}$ that extends to an SO$(n+1)$-invariant vector field on the link.
We compute that
$$\alpha(r) = ({\rm div}_{g_{0}}(h))(\partial_r) = -\sum (h(X_i, \nabla_{X_i}\partial_r) + h(Y_i, \nabla_{Y_i}\partial_r))  = -\frac{1}{r}({\rm tr}\, h )= 0,$$
where all derivatives are with respect to $g_0$ and where we have used {the fact} that $h(\partial_r, X) = 0$ for all vector fields $X$ on the cone. And similarly, we have $\beta(r)=0$, so that ${\rm div}_{g_0}(h)=0$, as desired.
\end{proof}

\begin{remark}\label{r:ct2} As discussed in Remark \ref{r:ct1}, the result of Proposition \ref{p:stenzelrate} indicates that the proof
of {\cite[Theorem 0.16]{Cheeger}} is not correct. (There is no problem with Theorems 0.13 or 0.15 {of \cite{Cheeger}}.) The proof in \cite{Cheeger} is valid if (7.106) and (7.137) hold. The Stenzel metric satisfies (7.106) but not (7.137). As a result, (7.136) and (7.138) also fail. In order to see this, one can compute {directly} that 
$$\dot{J}(\partial_r) = \dot{J}(J\partial_r) = 0, \quad X \perp \partial_r, J\partial_r \;\Longrightarrow \, \dot{J}X = 
\frac{1}{\tau}J\bar{X},$$
where the bar denotes complex conjugation on $\C^{n+1}$. Thus, $\dot{J} = \dot{J}^S$, so that $(d\dot{J}^S)^{SH} = 0$ (compare below). 
{We also} learn that (7.133) holds with $c = 2\frac{n}{n-1}$. Notice {here} that $J(r\partial_r) \perp \mathfrak{so}(n+1)$.

It turns out that the proof of Proposition 7.135 is correct, except for the fact that (7.134) should contain a correction term which vanishes when (7.137) holds. More precisely, the result cited from [B, p.~363] on p.~557 {of} \cite{Cheeger} is only applicable {whenever} (7.137) holds. In general, we have that
$$
\frac{1}{2}\Box^b = (\bar{\partial}^b)^*\bar{\partial}^b + \bar{\partial}^b(\bar{\partial}^b)^*  + (n-3)\mathcal{J}L_{\partial_\theta},
$$
where $\Box^b$ denotes the part of the Lichnerowicz Laplacian on the link that only involves horizontal covariant derivatives, as in (7.146). As a consequence, (7.151) should read $\underline{\mu} \geq -2(n-3)c$, so that (7.155) becomes a tautology. It is however still possible to prove that if (7.106) holds, then 
$$
(R^{\circ})^b \leq \lambda \,{\rm id}^b \;\Longrightarrow \, (n-3)c \leq\lambda.
$$

With regard to (7.106), we mention that while $(d \dot{J})^{SH}=0$
always holds, the remaining statements
can fail. However, the known counterexample {does in fact} satisfy the conclusions of Theorem 0.16. It may even be the case that $\dot{J} = \dot{J}^S$ always holds true, except in some special situations.
\end{remark}
\newpage
\appendix
\section{An $i\partial\bar{\partial}$-lemma for AC K{\"a}hler manifolds}\label{s:ddbar}

Recall that a complex manifold is \emph{$1$-convex} if it carries a smooth proper function which is strictly plurisubharmonic outside a compact subset.

\begin{example}
(i) AC K{\"a}hler manifolds are $1$-convex almost by definition.

(ii) Let $X$ be a compact complex manifold with a smooth divisor $D$ whose normal bundle $N_{D/X}$ is ample. Then $M = X \setminus D$ is $1$-convex; an exhaustion function can be constructed by transplanting $\frac{1}{h}$ from the total space of $N_{D/X}$ to $M$, where $h$ is a positively curved Hermitian metric on $N_{D/X}$.
\end{example}

The $1$-convexity of a complex manifold $M$ is equivalent to the existence of a \emph{Remmert reduction}, {that is}, a proper holomorphic map $\pi: M \to V$ onto a normal Stein variety $V$ with at worst finitely many singularities (in particular, $V \hookrightarrow \C^N$ and the singularities are algebraic) such that

$\bullet$ $\pi$ has connected fibers,

$\bullet$ ${\rm Exc}(\pi)$ is the maximal positive-dimensional compact analytic subset of $M$, 

$\bullet$ $\pi$ is an isomorphism onto its image away from ${\rm Exc}(\pi)$, {and}

$\bullet$ $\pi^*\mathcal{O}_{V} = \mathcal{O}_M$.

\noindent The map $\pi: M \to V$ is in fact unique up to isomorphism, and $V$ is then often simply referred to as
the Remmert reduction of $M$. See Grauert \cite[\S 2]{Grau:62} for details and further references.

The following proposition and corollary are due to van Coevering \cite[\S 4.2]{vanC}.

\begin{prop}\label{deldelbar}
Let $M$ be a $1$-convex manifold with trivial canonical bundle. 

{\rm (i)} We have $H^{k}(M,\mathcal{O}_{M})=0$ for each $k\geq 1$.

{\rm (ii)} If $n = \dim M > 2$ and if $K_R$ denotes the closed $R$-sublevel of a smooth proper function as in the definition of $1$-convex, then $H^{k}(M\setminus K_{R},\mathcal{O}_{M})=0$ for all $k\in\{1,..., n-2\}$ and $R \gg 1$.
\end{prop}

Part (ii) is a fairly delicate result:~According to \cite[Chapter V, \S 1.4]{theory}, $H^{n-d-1}(\C^n\setminus \C^d,\mathcal{O}_{\C^n}) \neq 0$, and $H^1(\C^2 \setminus 0, \mathcal{O}_{\C^2})$ is in fact infinite-dimensional; see also \cite[Theorem 8.9.2]{Joyce}.

\begin{corollary}\label{c:ideldelbar}
Let $M$ be an {\rm AC} K\"ahler manifold with trivial canonical bundle.

{\rm (i)} If $\alpha$ is an exact real $(1,1)$-form on $M$, then there exists $u\in C^{\infty}(M)$ such that $\alpha=i\p\bar{\p}u$.

{\rm (ii)} If $n = \dim M > 2$ and if $\alpha$ is an exact real $(1,1)$-form on $M\setminus K$ for some compact $K \subset M$, then there exist a compact $K' \supset K$ and $u\in C^{\infty}(M\setminus K')$ such that $\alpha=i\p\bar{\p}u$ on $M\setminus K'$. 
\end{corollary}

We now give the proof of Proposition \ref{deldelbar}. Part (i) is more standard and suffices for applications to quasiprojective manifolds \cite{ch2}.
 Part (ii) is only needed to prove Theorem \ref{thm:main} in full generality.

\begin{proof}[Proof of Proposition \ref{deldelbar}] Part (i) follows from a result of Grauert and Riemenschneider to the effect that $H^k(M, \mathcal{O}(K_M)) = 0$ for all $k \geq 1$ on every $1$-convex manifold $M$; see \cite[\S 2.4, Korollar]{Grauertriemannschneider}.

To prove Part (ii), let $\pi:M\to V$ denote the Remmert reduction of $M$ as above. We can clearly replace $M$ by $V$ in the assertion we need to prove. {Now, there} exists a long exact sequence
\begin{equation*}
H^{k-1}(V,\mathcal{O}_{V})\to H^{k-1}(V\setminus K_{R},\mathcal{O}_{V})\to H^{k}_{K_{R}}(V,\mathcal{O}_{V})\to H^{k}(V,\mathcal{O}_{V});
\end{equation*}
see \cite[Chapter I, \S1(a)]{Romanian} for the definitions. Since $V$ is Stein, this implies that
$$H^{k}(V\setminus K_{R},\mathcal{O}_{V})\cong H^{k+1}_{K_{R}}(V,\mathcal{O}_{V})$$ for $k\geq 1$.
We now wish to employ \cite[Chapter I, \S3, Theorem 3.1(a)]{Romanian} to see that $H^{k+1}_{K_{R}}(V,\mathcal{O}_{V})=0$ for $k \in\{0,...,n-2\}$, which would then clearly finish the proof.

Now, on one hand, $K_R$ can be written as the intersection of open sets with strictly pseudoconvex boundaries, which means that $K_R$ is a Stein compact in the sense of \cite[p.~29]{Romanian}.

The other condition we need to check is that the local rings of $V$ have depth $n$, or equivalently, that $V$ is Cohen-Macaulay. Rational implies Cohen-Macaulay \cite[Theorem 5.10]{KM}, and an isolated singularity is rational if there exists an $L^2$ holomorphic volume form in a deleted neighborhood \cite[Proposition 3.2]{Burns}. But this is clearly the case here because $K_M$ is trivial.
\end{proof}
\newpage

\bibliographystyle{amsplain}

\providecommand{\bysame}{\leavevmode\hbox to3em{\hrulefill}\thinspace}
\providecommand{\MR}{\relax\ifhmode\unskip\space\fi MR }
\providecommand{\MRhref}[2]{%
  \href{http://www.ams.org/mathscinet-getitem?mr=#1}{#2}
}
\providecommand{\href}[2]{#2}
\begin{thebibliography}{}

\end{thebibliography}


\begin{thebibliography}{10}

\bibitem{Altmann}
K.~Altmann, \emph{The versal deformation of an isolated toric {G}orenstein
  singularity}, Invent. Math. \textbf{128} (1997), 443--479.

\bibitem{Arezzo}
C.~Arezzo, A.~Ghigi, and G.~P.~Pirola, \emph{Symmetries, quotients and {K}\"ahler-{E}instein metrics}, J. Reine Angew. Math., \textbf{591} (2006), 177--200.
 
\bibitem{AF}
A.~Andreotti, T.~Frankel, \emph{The Lefschetz theorem on hyperplane
  sections}, Ann. of Math. (2) \textbf{69} (1959), 713--717.

\bibitem{Romanian}
C.~B{\u{a}}nic{\u{a}}, O.~St{\u{a}}n{\u{a}}{\c{s}}il{\u{a}}, \emph{Algebraic
  methods in the global theory of complex spaces}, Editura Academiei,
  Bucharest, 1976.

\bibitem{bm}
Y.~Bazaikin, E.~Malkovich, \emph{${\rm Spin(7)}$-structures on complex line
  bundles and explicit {R}iemannian metrics with the holonomy group ${\rm SU}(4)$}, Sb. Math. \textbf{202} (2011), 467--493, see also arXiv:1001.1622.

\bibitem{besse}
A.~Besse, \emph{Einstein manifolds}, Ergebnisse Math. Grenzgebiete (3),
  vol.~10, Springer-Verlag, Berlin, 1987.

\bibitem{BG}
O.~Biquard, P.~Gauduchon, \emph{Hyper-{K}{\"a}hler metrics on cotangent
  bundles of {H}ermitian symmetric spaces}, Geometry and Physics (Aarhus, 1995), Lecture Notes in Pure and Appl.
  Math., vol. 184, pp.~287--298, Dekker, New York, 1997, see also
  http://www.math.ens.fr/$\sim$biquard/pub.html.

\bibitem{borel}
A.~Borel, \emph{K{\"a}hlerian coset spaces of semisimple {L}ie groups}, Proc.
  Nat. Acad. Sci. USA \textbf{40} (1954), 1147--–1151.

\bibitem{BH}
A.~Borel, F.~Hirzebruch, \emph{Characteristic classes and homogeneous
  spaces, {\rm II}}, Amer. J. Math. \textbf{81} (1959), 315--382.

\bibitem{book:Boyer}
C.~Boyer, K.~Galicki, \emph{Sasakian geometry}, Oxford University Press,
  Oxford, 2008.

\bibitem{Burns}
D.~Burns, \emph{On rational singularities in dimensions {$>2$}}, Math. Ann.
  \textbf{211} (1974), 237--244.

\bibitem{Cal1}
E.~Calabi, \emph{M\'etriques k\"ahl\'eriennes et fibr\'es holomorphes}, Ann.
  Sci. \'Ecole Norm. Sup. (4) \textbf{12} (1979), 269--294.

\bibitem{delaossa}
P.~Candelas, X.~de~la Ossa, \emph{Comments on conifolds}, Nuclear Phys. B
  \textbf{342} (1990), 246--268.

\bibitem{Carron}
G.~Carron, \emph{On the quasi-asymptotically locally {E}uclidean geometry of
  {N}akajima's metric}, J. Inst. Math. Jussieu \textbf{10} (2011), 119--147.

\bibitem{Chan}
Y.-M. Chan, \emph{Desingularizations of {C}alabi-{Y}au $3$-folds with a conical
  singularity}, Q. J. Math. \textbf{57} (2006), 151--181.

\bibitem{ChNotes}
J.~Cheeger, \emph{Degeneration of {R}iemannian metrics under {R}icci curvature
  bounds}, Scuola Normale Superiore, Pisa, 2001.

\bibitem{Cheeger}
J.~Cheeger, G.~Tian, \emph{On the cone structure at infinity of {R}icci flat
  manifolds with {E}uclidean volume growth and quadratic curvature decay},
  Invent. Math. \textbf{118} (1994), 493--571.

\bibitem{cz}
T.~Christiansen, M.~Zworski, \emph{Harmonic functions of polynomial growth on certain complete
manifolds}, Geom. Funct. Anal. {\bf 6} (1996), 619--627.

\bibitem{RonanThesis}
R.~Conlon, \emph{On the construction of asymptotically conical {C}alabi-{Y}au
  manifolds}, Ph.D. thesis, Imperial College London, 2011.

\bibitem{ch2}
R.~Conlon, H.-J. Hein, \emph{Asymptotically conical {C}alabi-{Y}au
  manifolds, {\rm II}}, arXiv:1301.5312v2.

\bibitem{cvetic}
M.~Cveti{\v{c}}, G.~Gibbons, H.~L{\"u}, and C.~Pope, \emph{Ricci-flat metrics, harmonic forms and brane resolutions}, Comm. Math. Phys. {\bf 232} (2003), 457--500.

\bibitem{DP}
J.-P. Demailly, M.~Paun, \emph{Numerical characterization of the
  {K}{\"a}hler cone of a compact {K}{\"a}hler manifold}, Ann. of Math. (2)
  \textbf{159} (2004), 1247--1274.

\bibitem{donaldson:96}
S.~Donaldson, \emph{Symplectic submanifolds and almost-complex geometry}, J.
  Differential Geom. \textbf{44} (1996), 666--705.

\bibitem{donn}
H.~Donnelly, \emph{Harmonic functions on manifolds of nonnegative {R}icci curvature},
IMRN 2001, no. 8, 429--434.

\bibitem{Friedman}
R.~Friedman, \emph{Simultaneous resolution of threefold double points}, Math.
  Ann. \textbf{274} (1986), 671--689.

\bibitem{Goodman} J.~Goodman, \emph{Affine open subsets of algebraic varieties and ample divisors}, Ann. of Math. (2) \textbf{89} (1969), 160--183.

\bibitem{goto}
R.~Goto, \emph{{C}alabi-{Y}au structures and {E}instein-{S}asakian structures
  on crepant resolutions of isolated singularities}, J. Math. Soc. Japan {\bf 64} (2012), 1005--1052.

\bibitem{Grau:62}
H.~Grauert, \emph{\"{U}ber {M}odifikationen und exzeptionelle analytische
  {M}engen}, Math. Ann. \textbf{146} (1962), 331--368.

\bibitem{theory}
H.~Grauert, R.~Remmert, \emph{Theory of {S}tein spaces}, Springer-Verlag,
  Berlin, 2004.

\bibitem{Grauertriemannschneider}
H.~Grauert, O.~Riemenschneider, \emph{Verschwindungss\"atze f\"ur
  analytische {K}ohomologiegruppen auf komplexen {R}\"aum-en}, Invent. Math.
  \textbf{11} (1970), 263--292.

\bibitem{GH}
G.-M. Greuel, H.~Hamm, \emph{Invarianten quasihomogener vollst{\"a}ndiger
  {D}urchschnitte}, Invent. Math. \textbf{49} (1978), 67--86.

\bibitem{GriffithsHarris}
P.~Griffiths, J.~Harris, \emph{Principles of algebraic geometry}, John Wiley
  \& Sons Inc., New York, 1994.

\bibitem{Gross}
M.~Gross, \emph{Deforming {C}alabi-{Y}au threefolds}, Math. Ann. \textbf{308}
  (1997), 187--220.

\bibitem{Hamm}
H.~Hamm, \emph{Lokale topologische {E}igenschaften komplexer {R}{\"a}ume},
  Math. Ann. \textbf{191} (1971), 235--252.

\bibitem{Hein2}
H.-J. Hein, \emph{Weighted {S}obolev inequalities under lower {R}icci curvature
  bounds}, Proc. Amer. Math. Soc. \textbf{139} (2011), 2943--2955.

\bibitem{Joyce}
D.~Joyce, \emph{Compact manifolds with special holonomy}, Oxford University
  Press, Oxford, 2000.

\bibitem{Kas}
A.~Kas, M.~Schlessinger, \emph{On the versal deformation of a complex space
  with an isolated singularity}, Math. Ann. \textbf{196} (1972), 23--29.

\bibitem{KM}
J.~Koll{\'a}r, S.~Mori, \emph{Birational geometry of algebraic varieties},
  Cambridge University Press, Cambridge, 1998.

\bibitem{Kronheimer}
P.~Kronheimer, \emph{The construction of {\rm ALE} spaces as hyper-{K}\"ahler
  quotients}, J. Differential Geom. \textbf{29} (1989), 665--683.

\bibitem{Lebrun}
C.~LeBrun, \emph{Fano manifolds, contact structures, and quaternionic
  geometry}, Internat. J. Math. \textbf{6} (1995), 419--437.

\bibitem{li}
P.~Li, \emph{Harmonic functions of linear growth on {K}{\"a}hler manifolds of
  nonnegative {R}icci curvature}, Math. Res. Lett. \textbf{2} (1995), 79--94.

\bibitem{liwang}
P.~Li, J.~Wang, \emph{Comparison theorem for {K}{\"a}hler manifolds and
  positivity of spectrum}, J. Differential Geom. \textbf{69} (2005), 43--74.

\bibitem{Lockhart}
R.~Lockhart, R.~McOwen, \emph{Elliptic differential operators on noncompact
  manifolds}, Ann. Scuola Norm. Sup. Pisa Cl. Sci. (4) \textbf{12} (1985),
  409--447.

\bibitem{mal}
E.~Malkovich, \emph{On new explicit {R}iemannian metrics with holonomy group ${\rm SU}(2(n+1))$}, Sibirsk. Mat. Zh. {\bf 52} (2011), 95--99, see also arXiv:1010.2590.

\bibitem{Marshall}
S.~Marshall, \emph{Deformations of special {L}agrangian submanifolds}, Ph.D.
  thesis, University of Oxford, 2002, available at
  {http://people.maths.ox.ac.uk/joyce/theses/MarshallDPhil.pdf}.

\bibitem{Sparks2}
D.~Martelli, J.~Sparks, \emph{Resolutions of non-regular {R}icci-flat
  {K}\"ahler cones}, J. Geom. Phys. \textbf{59} (2009), 1175--1195.

\bibitem{mintaubnut}
V.~Minerbe, \emph{Rigidity for multi-Taub-{\rm NUT} metrics}, J. Reine Angew. Math. {\bf 656} (2011), 47--58.

\bibitem{Pac}
T.~Pacini, \emph{Desingularizing isolated conical singularities:~uniform estimates via weighted Sobolev spaces}, Comm. Anal. Geom. {\bf 21} (2013), 105--170.

\bibitem{r}
F.~Reidegeld, \emph{Exceptional holonomy and {E}instein metrics constructed
  from {A}loff-{W}allach spaces}, Proc. Lond. Math. Soc. (3) \textbf{102}
  (2011), 1127--1160.

\bibitem{AG5}
I.~Shafarevich (ed.), \emph{Algebraic geometry, {\rm V}, Fano varieties}, Encyclopaedia of
  Mathematical Sciences, vol.~47, Springer-Verlag, Berlin, 1999.

\bibitem{Sparks-survey}
J.~Sparks, \emph{Sasaki-{E}instein manifolds}, Surveys in Differential
  Geometry, vol.~16, pp.~265--324, International Press, Somerville, 2011, see
  also arXiv:1004.2461.

\bibitem{Stenzel}
M.~Stenzel, \emph{Ricci-flat metrics on the complexification of a compact rank
  one symmetric space}, Manuscripta Math. \textbf{80} (1993), 151--163.

\bibitem{Tian}
G.~Tian, S.-T. Yau, \emph{Complete {K}\"ahler manifolds with zero {R}icci
  curvature, {\rm II}}, Invent. Math. \textbf{106} (1991), 27--60.

\bibitem{vanC}
C.~van Coevering, \emph{A construction of complete {R}icci-flat {K}\"ahler
  manifolds}, arXiv:0803.0112v5.

\bibitem{vanC3}
\bysame, \emph{Regularity of asymptotically conical {R}icci-flat {K}\"ahler
  metrics}, arXiv:0912.3946v5.

\bibitem{vanC2}
\bysame, \emph{Ricci-flat {K}\"ahler metrics on crepant resolutions of
  {K}\"ahler cones}, Math. Ann. \textbf{347} (2010), 581--611.

\bibitem{vanC4}
\bysame, \emph{Examples of asymptotically conical {R}icci-flat {K}{\"a}hler
  manifolds}, Math. Z. \textbf{267} (2011), 465--496.

\bibitem{voisin2}
C.~Voisin, \emph{Th{\'e}orie de {H}odge et g{\'e}om{\'e}trie alg{\'e}brique
  complexe}, Soci{\'e}t{\'e} Math{\'e}matique de France, Paris, 2002.

\end{thebibliography}

\end{document}